\documentclass[11pt]{amsart}

\usepackage{amssymb, latexsym, amsfonts, pigpen, enumitem, mathrsfs, amsthm, bbm, stackrel, etoolbox, color, tikz, tikz-cd, amsmath, nicefrac, eucal}
\usepackage[matrix,arrow,curve]{xy}
\usepackage[colorlinks=true,pagebackref=true,citecolor=cyan,linkcolor=magenta]{hyperref}
\hypersetup{linktocpage}
\usepackage[capitalise]{cleveref}

\usepackage{verbatim}
\usepackage[letterpaper, top=0.9in, bottom=0.9in, left=0.9in, right=0.9in]{geometry}

\usetikzlibrary{matrix,arrows}
\usepackage{stmaryrd}
\usepackage{listings}

\numberwithin{equation}{section}

\newtheorem{theorem}[equation]{Theorem}
\newtheorem{lemma}[equation]{Lemma}
\newtheorem{proposition}[equation]{Proposition}
\newtheorem{corollary}[equation]{Corollary}
\newtheorem{theoremintr}{Theorem}

\newtheorem{propintr}{Proposition}

\theoremstyle{definition}
\newtheorem{definition}[equation]{Definition}
\newtheorem{example}[equation]{Example}
\newtheorem{notation}[equation]{Notation}
\newtheorem{remark}[equation]{Remark}
\newtheorem{construction}[equation]{Construction}

\newcommand{\ev}{{\mathrm{ev}}}
\newcommand{\Sp}{\mathbf{Sp}}
\newcommand{\realiz}{\mathrm{Re}}

\newcommand{\ret}{\mathrm{r\acute{e}t}}
\newcommand{\MUtf}{\mathrm{M}\tilde{\mathrm{U}}}
\newcommand{\slice}{\mathrm{s}}

\newcommand{\Jhom}{\mathrm{J}}
\newcommand{\mot}{\mathrm{mot}}
\newcommand{\Spc}{\mathbf{Spc}}
\newcommand{\SmS}{\mathbf{Sm}_S}
\newcommand{\Sm}{\mathbf{Sm}}
\newcommand{\SHS}{\mathbf{SH}(S)}
\newcommand{\sph}{\mathbbm{1}}
\newcommand{\Th}{{\mathrm{Th}}}

\newcommand{\SL}{\mathrm{SL}}
\newcommand{\SLc}{\mathrm{SL}^c}
\newcommand{\ML}{\mathrm{ML}}
\newcommand{\GL}{\mathrm{GL}}
\newcommand{\Krk}{\mathrm{K_{rk=0}}}

\newcommand{\KSLrk}{\mathrm{K^{SL}_{rk=0}}}
\newcommand{\PShv}{\mathrm{PShv}}

\newcommand{\C}{\mathbb{C}}
\newcommand{\R}{\mathbb{R}}
\newcommand{\BU}{\mathrm{BU}}
\newcommand{\Eone}{\mathbb{E}_1}
\newcommand{\Einf}{\mathbb{E}_\infty}

\newcommand{\Gr}{{\mathrm{Gr}}}
\newcommand{\SGr}{{\mathrm{SGr}}}
\newcommand{\MGr}{{\mathrm{MGr}}}
\newcommand{\Proj}{\mathbb{P}}
\newcommand{\Z}{{\mathbb{Z}}}
\newcommand{\MGL}{\mathrm{MGL}}
\newcommand{\MSL}{\mathrm{MSL}}
\newcommand{\MML}{\mathrm{MML}}
\newcommand{\MSp}{\mathrm{MSp}}
\newcommand{\MSU}{\mathrm{MSU}}
\newcommand{\MU}{\mathrm{MU}}

\newcommand{\GW}{\mathrm{GW}}
\newcommand{\Vect}{{\rm Vect}}
\newcommand{\sVect}{{\rm sVect}}
\newcommand{\W}{\mathrm{W}}
\newcommand{\K}{\mathrm{K}}
\newcommand{\A}{\mathbb{A}}

\newcommand{\id}{\operatorname{id}}
\newcommand{\colim}{\operatorname*{colim}}
\newcommand{\limit}{\operatorname*{lim}}
\newcommand{\cofib}{\operatorname*{cofib}}
\newcommand{\struct}{\mathcal{O}}
\newcommand{\Gm}{{\mathbb{G}_m}}
\newcommand{\SH}{\mathbf{SH}}
\newcommand{\Spec}{\operatorname{Spec}}
\newcommand{\Sper}{\operatorname{Sper}}
\newcommand{\Q}{\mathbb{Q}}
\newcommand{\Pic}{\operatorname{Pic}}
\newcommand{\et}{\mathrm{\acute{e}t}}
\newcommand{\EE}{\mathrm{E}}
\newcommand{\etatop}{\eta_{\mathrm{top}}}

\newcommand{\bigslant}[2]{{\left.\raisebox{.2em}{$#1$}\middle/\raisebox{-.2em}{$#2$}\right.}}

\makeatletter
\DeclareMathOperator{\thom}{th}

\title{On the metalinear algebraic cobordism spectrum}
\author{Ahina Nandy}
\address{Radboud Universiteit, Mathematical Institute, Postbus 9010, 6500 GL Nijmegen, Netherlands}
\email{\href{mailto:ahina.nandy@ru.nl}{ahina.nandy@ru.nl}}
\author{Egor Zolotarev}
\address{LMU M\"unchen, Mathematisches Institut, Theresienstr. 39, 80333 M\"unchen, Germany}
\email{\href{mailto:zolotarev@math.lmu.de}{zolotarev@math.lmu.de}, \href{mailto:zolotarev-egv@yandex.ru}{zolotarev-egv@yandex.ru}}
\urladdr{\href{https://sites.google.com/view/egor-zolotarev}{https://sites.google.com/view/egor-zolotarev}}
\keywords{Motivic homotopy theory, metalinear algebraic cobordism, special linear algebraic cobordism}
\subjclass[2020]{Primary 14F42; Secondary 55P42, 57R90}

\begin{document}
\begin{abstract}
In this paper, we study the metalinear algebraic cobordism spectrum $\mathrm{MML}$ (also sometimes denoted $\mathrm{MSL}^c$), which is built from the structure groups of oriented vector bundles. We establish an interpolation between $\mathrm{MSL}$ and $\mathrm{MML}$ and deduce that the canonical morphism $\mathrm{MSL}\to \mathrm{MML}$ admits a retraction. We parametrize all such retractions in the category of $\mathrm{MSL}$-modules and, after fixing one of them, obtain an equivalence $\mathrm{MML}\cong\mathrm{MSL}\oplus \Sigma^{2,1}\mathrm{MGL}$. As an application of these results, we determine various invariants of the metalinear algebraic cobordism spectrum over a field (after inverting the exponential characteristic). More precisely, we determine the first few Milnor--Witt stems of $\mathrm{MML}$ in terms of the very effective algebraic and hermitian K-theory spectra, and the geometric diagonal of $\mathrm{MML}$ in terms of Stong's complex-spin cobordism ring. We also compute the slices and use them to describe the category of $2$-inverted modules over the $\mathbb{E}_\infty$-ring spectrum $\mathrm{MML}$.
\end{abstract}
\maketitle
\tableofcontents
\section{Introduction}
The algebraic cobordism spectrum $\MGL$ is an object of the Morel--Voevodsky $\Proj^1$-stable $\A^1$-homotopy category that was introduced by Vladimir Voevodsky \cite{Voev98} for his original approach to the proof of Milnor's conjecture on Galois cohomology \cite{Voev_MC} (note that $\MGL$ was removed in the published proof \cite{Voevodsky_MC}). The construction of this spectrum is similar to the one of the unoriented cobordism and complex cobordism spectra in topology. 
Since then, this spectrum has found a special role in motivic homotopy theory. One reason for this is its universal property among oriented (i.e., those that admit Thom classes for vector bundles) motivic homotopy commutative ring spectra \cite{PPR08}. The universal property is reflected in the structure of the formal group law associated with $\MGL$, which is known to be the universal one at least over fields (after inverting the exponential characteristic) by the work of Hopkins--Morel--Hoyois \cite{Hoy15}. This result might be viewed as a motivic version of the celebrated theorem of Daniel Quillen \cite{Quillen}. Classical examples of motivic spectra, such as algebraic K‑theory and motivic cohomology, are endowed with canonical orientations and fit naturally into the framework provided by $\MGL$.

In recent years, people have intensively studied ``quadratic'' analogues of the usual motivic cohomology theories, which are not oriented in the usual sense. Examples of such theories include hermitian K-theory, Balmer--Witt theory (also known as L-theory), Milnor--Witt motivic cohomology (and Chow--Witt groups as its geometric diagonal), and cohomology represented by a homotopy module. In this context, it is natural to ask what the quadratic analogue of algebraic cobordism is -- a theory that would play the role of $\MGL$ for these spectra. There are at least two reasonable answers to this question that are closely related.
\subsection{Special linear vs. metalinear algebraic cobordism}
One candidate is called \textit{special linear algebraic cobordism} and denoted by $\MSL$. It was introduced by Ivan Panin and Charles Walter \cite{PW-Thom} via a construction similar to that of $\MGL$, but using the special linear groups instead of the general linear ones. However, while $\MSL$ admits Thom classes for vector bundles with trivialized determinant (which we call \textit{strictly oriented vector bundles}, following Morel--Sawant \cite[Definition 2.14]{MS_Cellular}), all the above examples of non-orientable motivic spectra (apart from $\MSL$ itself) admit a more general theory of Thom classes. More precisely, they admit multiplicative Thom classes for vector bundles with a chosen square root of the determinant (which we call \textit{oriented vector bundles}, following Fabien Morel \cite[Definition 4.3]{Morel_book}). One can define the structure groups of such vector bundles (we call them \textit{metalinear groups}, following Asok--Hoyois--Wendt \cite[Remark 3.3.2]{AHW_BG}) and construct the corresponding \textit{metalinear algebraic cobordism spectrum} $\MML$ (see \cite[Remark 7.11]{HJNY-HermKtheory}, \cite[\S 6.3]{BrazeltonWendtMSLc}, \cite[\S 9]{Hoyois_Land}). This is the main object of study in this paper.

By construction, there is a forgetful morphism of motivic $\Einf$-ring spectra $\MSL\to \MML$. Fabien Morel, in his talk at the Vladimir Voevodsky Memorial Conference, asked whether this morphism is an equivalence\footnote{Strictly speaking, he mentioned this question for the corresponding geometric cobordism theories, which are still conjectural. It is natural to expect that these theories, if they exist, would satisfy Pontryagin--Thom theorems (at least over fields of characteristic zero), similar to the one proved by Marc Levine for Levine--Morel algebraic cobordism $\Omega^*$ and $\MGL$ \cite{Levine}. If this is the case, Morel's question is weaker than the corresponding question for the associated spectra. Nevertheless, Theorem~\ref{theorem_A} implies that it would have a negative answer as well.}. The following theorem shows that this is not the case:
\begin{theoremintr}\label{theorem_A}
    Let $S$ be a qcqs scheme. The forgetful morphism $\MSL\to \MML$ is included into the following (split) cofiber sequence
    $$ \MSL\to \MML\xrightarrow{\partial} \Sigma^{2,1} \MGL\xrightarrow{0} \Sigma^{1,0}\MSL $$
    of $\MSL$-modules over $S$,
    where $\partial$ is the unique $\MSL$-linear operation that corresponds to the characteristic class $[-1]_F(c_1(\sqrt{\det\gamma_\infty^{\ML}}))$ under the Thom isomorphism, see Lemmas \ref{lem:MGL(MML)} and \ref{lem:MGL(MGr)}. The set of homotopy classes of splittings of this cofiber sequence (in the stable category of $\MSL$-modules) is in bijection with the set
    $$ \{ t\in \MSL^{2,1}(\Th_{\Proj^\infty}(\struct(-2)))\,|\,t\vert_{\Proj^0}=\Sigma^{2,1} 1 \}. $$
\end{theoremintr}
This set can be interpreted as the set of possible choices of normalized Thom classes for the universal oriented vector bundle of rank $1$ in $\MSL$-cohomology. Of course, if we fix such a Thom class $t$, then we obtain an equivalence of $\MSL$-modules:
\[ \begin{pmatrix}
    t \\ \partial
\end{pmatrix}\colon \MML\xrightarrow{\simeq} \MSL\oplus \Sigma^{2,1}\MGL, \]
which allows us to reduce many questions about $\MML$ to those about $\MSL$ and $\Sigma^{2,1}\MGL$.
The proof of the above theorem proceeds in two steps. The first step is the interpolation between $\MSL$ and $\MML$, which states that $\MML$ is the free $\MSL$-module on $\Sigma^{-2,-1}\Th_{\Proj^\infty}(\struct(-2))$, see Theorem~\ref{thm:mslc_free}. The second one uses the cancellation of squares of line bundles in the stable $\infty$-category of $\MSL$-modules, which is based on a non-linear equivalence of Thom spaces due to Alexey Ananyevskiy \cite[Lemma~4.1]{AnSL} (see also the earlier observation of Oliver R\"ondigs \cite[Lemma 4.2]{Oliver_theta}). In fact, this cancellation gives a specific choice of $t$, but we use it only for the computation of the cofiber because it is unclear whether this particular choice is better than the others. 
Along the way, we answer a question of Alexey Ananyevskiy about the relation between $\SL$- and $\ML$-oriented spectra asked in \cite[Remark 4.4]{AnSL}, see Remark \ref{ques:twisted_thom_mult}. We also show that there is no element $t$ giving rise to a multiplicative retraction $\MML\to\MSL$, even in a weak sense, see Remark~\ref{rem:mult_splitting}.
\subsection{Applications}
We then determine several invariants of $\MML$ when $S=\Spec(F)$ for a field $F$. To the best of our knowledge, all these results are new.

First, we compute the zeroth line of the homotopy groups of $\MML$. A similar result for $\MSL$ follows from an easy connectivity argument that is based on $\SL_1=1$, see \cite[Remark~16.35]{BH-Norms}.
\setcounter{propintr}{1}
\begin{propintr}\label{proposition_B}
    Let $F$ be a field. The unit map $\sph\to \MML$ induces an isomorphism of graded rings
    \[ \bigoplus_{n\in\Z}\pi_{n,n}(\MML)\cong \K^{\mathrm{MW}}_{-*}(F), \]
    where the right hand side is the Milnor--Witt K-theory of $F$, see \cite[Definition 1.21]{Morel_book}.
\end{propintr}
After that we describe the Milnor--Witt stems $\bigoplus_{n\in\Z}\pi_{n+m,n}(\MML)$ for $1\leq m \leq 3$ (away from the exponential characteristic of $F$) in terms of the very effective hermitian K‑theory $\mathrm{kq}$ and (very) effective algebraic K-theory $\mathrm{kgl}$ (see Propositions \ref{prop:pi_1_MML}, \ref{prop:pi_2_MML}, \ref{prop:pi_3_MML}). These results use Theorem \ref{theorem_A} together with the analogous computations for $\MSL$ presented in \cite{Ahina, MSL-slices}.

Second, we determine the geometric diagonal (also known as the Chow zero line) of the homotopy groups of $\MML$. This graded ring is of particular interest, because it contains classes of oriented smooth projective varieties over $F$ via a motivic version of the Pontryagin--Thom construction, see \cite[\S 3]{LYZ} for a discussion of classes of Calabi--Yau varieties in the geometric diagonal of $\MSL$. The approach here is very close to the one presented by the second author in \cite{Egor} for $\MSL$. As in \textit{loc.cit.}, the answer is expressed by gluing the topological counterpart and the $\eta$-periodic part via a pullback square. To establish the topological part, we compute the complex realization of $\MML$, which turns out to be Stong's complex-spin cobordism spectrum $\mathrm{M}\Sigma$ \cite{Stong} (i.e., the Thom spectrum corresponding to the cobordism theory of manifolds equipped with compatible stable unitary and stable spin structures), see also Appendix \ref{appendix_Stong} for an overview. For the following theorem, we stress that the homotopy groups of $\mathrm{M}\Sigma$ are known (see Theorem \ref{theorem:homotopy_groups_MSigma}), whereas the ring structure of the coefficient ring is not well understood.
\setcounter{theoremintr}{2}
\begin{theoremintr}\label{theorem_C}
    Let $F$ be a field of characteristic zero. Then the following diagram is a pullback square of graded $\GW(F)$-algebras:
        \[\begin{tikzcd} {\pi_{2*,*}(\MML)} \arrow[dr, phantom, "\lrcorner", very near start] \arrow[dd, two heads] \arrow[rr, two heads] & & {\pi_{2*}(\mathrm{M}\Sigma)} \arrow[dd, two heads] \\  & {} & & \\{\W(F)[u_4,u_8,\dots]} \arrow[rr, "\overline{\mathrm{rk}}", two heads] &  & {\Z/2[u_4,u_8,\dots]},\end{tikzcd}\]
    where the left vertical map is the quotient by the annihilator of $\eta\in\pi_{1,1}(\MML)$ (see Lemma \ref{lemma:quotient_by_eta_tors}), the right one is the quotient by the annihilator of $\etatop\in\pi_{1}(\mathrm{M}\Sigma)$, and the top horizontal one is described in Lemma \ref{lemma:quotient_by_fund_ideals}. If $F$ is a field of characteristic $p>0$, then the same result holds after inverting $p$.
\end{theoremintr}
Next, we compute the slices of $\MML$. Below we denote by $\mathrm{M}A\in\SH(F)$ the motivic Eilenberg--MacLane spectrum associated with an abelian group $A$. The following proposition is the only fact in the paper that requires a choice of an element $t$ from Theorem \ref{theorem_A} in the proof.
\setcounter{propintr}{3}
\begin{propintr}\label{proposition_D}
    Let $F$ be a field of characteristic zero. Then the slices of $\MML$ are given by
    \[ \slice_q(\MML)\cong \bigoplus_{0\leq s<\nicefrac{q}{2}} \Sigma^{q+2s,q}\mathrm{M}(\Z/2)^{p(\lfloor\nicefrac{s}{2}\rfloor)}\oplus \Sigma^{2q,q}\mathrm{M}\Z^{p(q)}, \]
    where $p(-)$ denotes the partition function and $\lfloor-\rfloor$ denotes the floor function.
    If $F$ is a field of characteristic $p>0$, then the same result holds after inverting $p$.
\end{propintr}
The above formula has an interpretation in terms of the $E_2$-page of the Adams--Novikov spectral sequence for $\mathrm{M}\Sigma$, see Remark \ref{remark:slices_AN}.
After computing the slices, we prove that the morphism $\MML[\nicefrac{1}{2}]\to \MGL[\nicefrac{1}{2}]$ induces an isomorphism on all slices (here we omit inversion of the characteristic for simplicity) and deduce that it gives an equivalence of motivic $\Einf$-ring spectra $\MML[\nicefrac{1}{2}]^+\cong \MGL[\nicefrac{1}{2}]$. We then categorify this comparison, obtaining a description of the plus part of the category of $2$-inverted modules over the motivic $\Einf$-ring spectrum $\MML$. After that, we use Bachmann's theorem on real \'etale realization to describe the minus part. This leads to the following theorem:
\setcounter{theoremintr}{4}
\begin{theoremintr}\label{theorem_E}
    Let $F$ be a field of characteristic zero. Then there is an equivalence of symmetric monoidal $\infty$-categories:
    \[ \mathrm{Mod}_{\MML}(\SH(F))[\nicefrac{1}{2}]\cong \mathrm{Mod}_{\MGL}(\SH(F))[\nicefrac{1}{2}]\times \mathrm{Mod}_{\underline{\mathrm{MSO}}}(\Sp(\Sper(F)))[\nicefrac{1}{2}]. \]
    Here $\Sp(\Sper(F))$ is the stable $\infty$-category of sheaves of spectra on the topological space $\Sper(F)$ (see Appendix \ref{appendix_real_etale} for an overview) and $\underline{\mathrm{MSO}}$ is the constant sheaf of $\Einf$-rings associated with the (topological) special orthogonal cobordism spectrum $\mathrm{MSO}\in \mathrm{CAlg}(\Sp)$. If $F$ is a field of characteristic $p>0$, then the same result holds after inverting $p$.
\end{theoremintr}
This theorem can be viewed as the cobordism refinement of the description of the Milnor--Witt motives with $\Z[\nicefrac{1}{2}]$-coefficients:
\[ \widetilde{\mathbf{DM}}(F,\Z[\nicefrac{1}{2}])\cong \mathbf{DM}(F,\Z[\nicefrac{1}{2}])\times \mathbf{D}(\Sper(F),\Z[\nicefrac{1}{2}])\]
obtained in \cite[Theorem 5.0.2]{Milnor_Witt_motives} and \cite[Proposition 41]{Bachmann_ret}. Here $\mathbf{DM}(F,\Z[\nicefrac{1}{2}])$ is Voevodsky's category of motives with $\Z[\nicefrac{1}{2}]$-coefficients, and $\mathbf{D}(\Sper(F),\Z[\nicefrac{1}{2}])$ is the derived category of sheaves of $\Z[\nicefrac{1}{2}]$-modules on the topological space $\Sper(F)$.

We also deduce a decomposition of the rational metalinear algebraic cobordism spectrum $\MML\otimes\Q$ over an arbitrary qcqs scheme $S$ into a product of polynomial homotopy commutative ring spectra over rational motivic cohomology and rational Witt motivic cohomology, see Remark \ref{remark:rat_decompos}.

\subsection{Outline}
The beginning of each section contains more detailed information about its contents. In Section \ref{section:2}, we recall the basics on metalinear groups and oriented vector bundles, and define the corresponding K-theory presheaf. In Section \ref{section:3}, we define the metalinear algebraic cobordism spectrum and discuss its basic properties. In Section \ref{section:4}, we prove the interpolation between $\MSL$ and $\MML$, and establish Theorem \ref{theorem_A}. This is the main technical part of the paper. In Section \ref{section:5}, we prove Proposition \ref{proposition_B}, compute the first Milnor--Witt stems of $\MML$ and its geometric diagonal, thereby proving Theorem \ref{theorem_C} (in characteristic different from $2$). In Section \ref{section:6}, we prove Proposition \ref{proposition_D}, compare the plus part of $\MML[\nicefrac{1}{2}]$ with $\MGL[\nicefrac{1}{2}]$, and prove Theorem \ref{theorem_E}.

In Appendix \ref{appendix_Stong}, we give an overview of Stong's complex-spin cobordism spectrum and prove that it is equivalent to the Thom spectrum constructed from the $2$-fold coverings of unitary groups. In Appendix \ref{appendix_real_etale}, we recall the construction of the real \'etale realization and compute it for motivic $\Einf$-ring spectra defined over $\Spec(\Z)$. In Appendix \ref{appendix_char2}, we compute the geometric diagonals of $\MSL$ and $\MML$ over fields of characteristic $2$ away from $2$ (in particular, there we prove Theorem \ref{theorem_C} in this case).
\subsection{Notations and conventions}
We use the language of higher category theory following Lurie's books \cite{LHTT} and \cite{LHA}. We consider $1$-categories as $\infty$-categories via the nerve construction. The notation $\Spc$ stands for the $\infty$-category of spaces (or anima) and $\Sp$ stands for the $\infty$-category of spectra. We denote by $\mathrm{CAlg}(\mathcal{C}^\otimes)$ (or simply $\mathrm{CAlg}(\mathcal{C})$) the $\infty$-category of $\Einf$-algebras in a symmetric monoidal $\infty$-category $\mathcal{C}^\otimes$. For $\EE\in \mathrm{CAlg}(\mathcal{C})$, we denote by $\mathrm{Mod}_\EE(\mathcal{C})$ the $\infty$-category of $\EE$-modules in $\mathcal{C}$, see \cite[\S 4.5]{LHA}.

Throughout the text $S$ denotes a base scheme, which is assumed to be quasi-compact and quasi-separated. We denote by $\SmS$ the category of qcqs smooth $S$-schemes, by $\PShv(\SmS)$ the $\infty$-category of presheaves (of spaces) on $\SmS$, and by $\SHS$ the $\Proj^1$-stable $\A^1$-homotopy $\infty$-category over $S$. Given a presheaf $X\in \PShv(\SmS)$ (resp. pointed presheaf $(X,x)$), we denote by $\Sigma^\infty_{\Proj^1}X_+$ (resp. $\Sigma^\infty_{\Proj^1}(X,x)$ or simply $\Sigma^\infty_{\Proj^1}X$) the associated motivic spectrum, omitting the motivic localization.

We use the following conventions about motivic spheres and suspension functors: $\Sigma^{p,q}$ denotes the smash product with the motivic sphere $S^{p,q}:=S^{p-q}\wedge \mathbb{G}_m^{\wedge q}$, and $\pi_{p,q}(\EE)$ denotes the corresponding homotopy group of a spectrum $\EE\in\SHS$. We also use the notation $\Sigma^q_{\Proj^1}$ for the $q$-th $\Proj^1$-suspension functor, which is canonically equivalent to $\Sigma^{2q,q}$. We write $\Omega^q_{\Proj^1}$ for the $q$-th $\Proj^1$-loop functor, which is nothing but the inverse autoequivalence of $\Sigma^q_{\Proj^1}$. 

For a vector bundle $\mathcal{E}$ over a smooth $S$-(ind)-scheme $X$, we denote by $\Th_X(\mathcal{E})\in \PShv(\SmS)$ its Thom space $\mathcal{E}/\mathcal{E}^\circ$, where $\mathcal{E}^\circ$ is the complement of the zero section of $\mathcal{E}$. By standard abuse of notation, we also denote by $\Th_X(\mathcal{E})$ the associated motivic spectrum $\Sigma^\infty_{\Proj^1}\Th_X(\mathcal{E})=\Sigma^\infty_{\Proj^1}(\mathcal{E}/\mathcal{E}^\circ,\mathcal{E}^\circ)$.
\subsection{Acknowledgments}
We would like to thank Alexey Ananyevskiy for fruitful discussions and for careful reading of a draft of this paper, Matthias Wendt for discussions and helpful comments, Tom Bachmann for answering a question about the real \'etale realization, and Julie Bannwart for sharing the latest version of her preprint before it appeared on arXiv.
Egor Zolotarev is supported by the DFG research grant AN 1545/4-1.
\section{Preliminaries on metalinear groups and oriented vector bundles}\label{section:2}
In this section, we discuss basics of metalinear groups, their torsors, oriented vector bundles, and various K-theory presheaves. The most important object for the subsequent sections is the presheaf of nonunital $\Einf$-rings $\K^\ML_\ev$, which is the K-theory of oriented vector bundles (of even rank).
\begin{definition}
Following \cite[Remark 3.3.2]{AHW_BG}, we define the \textit{metalinear group} $\ML_n$ to be the kernel of the following epimorphism of linear algebraic groups over $S$: \begin{equation}\label{def_slnc} \GL_n\times\Gm\to \Gm,\ (g,t)\mapsto t^{-2}\cdot\det(g). \end{equation} 
The original reference for this definition is \cite[\S 3]{PW-BO}, where this group is denoted by $\SLc_n$.
\end{definition}
We are interested in $\ML_n$-torsors. The key point is that this notion does not depend on the choice of the Grothendieck topology (within reasonable limits).
\begin{lemma}[{\cite[Proposition 3.3]{BrazeltonWendtMSLc}}]\label{lemma:ML_special}
    The linear algebraic group $\ML_n$ is special in the sense of Serre, i.e., every fppf $\ML_n$-torsor over a (not necessarily smooth) $S$-scheme is locally trivial in the Zariski topology.
\end{lemma}
\begin{remark}
    For $\tau\in\{\mathrm{Zar},\,\mathrm{Nis},\,\et\}$, we denote by $\mathrm{B}_\tau \ML_n$ the $\tau$-classifying space that sends a smooth $S$-scheme to the groupoid of $\tau$-locally trivial $\ML_n$-torsors.
    It then follows from Lemma \ref{lemma:ML_special} that the natural morphisms of the classifying spaces \[\mathrm{B_{Zar}}\ML_n\to \mathrm{B_{Nis}}\ML_n\to \mathrm{B}_{\et}\ML_n\] are equivalences of presheaves on $\SmS$. This allows us to use the symbol $\mathrm{B}\ML_n$ without specifying the underlying Grothendieck topology. We use this convention for other special groups as well.
\end{remark}
The metalinear groups appear as structure groups of oriented vector bundles.
\begin{definition}[{\cite[Definition 4.3]{Morel_book}}]
    A \textit{rank $n$ oriented vector bundle} over a scheme $X$ is a triple $(\mathcal{E},\mathcal{L},\lambda)$ consisting of a vector bundle $\mathcal{E}$ over $X$ of rank $n$, a line bundle $\mathcal{L}$ over $X$, and an isomorphism of line bundles $\lambda\colon\det(\mathcal{E}) \xrightarrow{\simeq} \mathcal{L}^{\otimes 2}$. Isomorphisms of oriented vector bundles are defined in the expected way.
\end{definition}
The direct sum $(\mathcal{E},\mathcal{L},\lambda)\oplus(\mathcal{E}',\mathcal{L}',\lambda')$ of oriented vector bundles $(\mathcal{E},\mathcal{L},\lambda)$ and $(\mathcal{E}',\mathcal{L}',\lambda')$ is given by the triple $(\mathcal{E}\oplus \mathcal{E}',\mathcal{L}\otimes \mathcal{L}',\lambda\otimes\lambda')$. Here, by abuse of notation $\lambda\otimes\lambda'$ denotes the following isomorphism \[\det(\mathcal{E}\oplus \mathcal{E}')\cong\det(\mathcal{E})\otimes\det(\mathcal{E}')\xrightarrow{\lambda\otimes\lambda'}\mathcal{L}^{\otimes2}\otimes(\mathcal{L}')^{\otimes2}\cong (\mathcal{L}\otimes \mathcal{L}')^{\otimes 2}.\] 
The trivial rank $n$ oriented vector bundle over $X$ is an oriented vector bundle that is isomorphic to $(\struct_X,\struct_X,\theta)^{\oplus n}$, where $\theta$ is the canonical isomorphism $\theta\colon\struct_X\xrightarrow{\simeq}\struct_X^{\otimes2}$. According to \cite[Lemma~2.7]{AnSL}, oriented vector bundles are locally trivial in the Zariski topology. 
\begin{notation}
    We denote by $\Vect_n^\ML(X)$ the groupoid of oriented vector bundles over $X$, which is canonically equivalent to the groupoid of $\ML_n$-torsors $\mathrm{B}\ML_n(X)$ by \cite[Remark 2.8]{AnSL}.
\end{notation}
\begin{remark}
    This is the reason why oriented vector bundles are also sometimes called ``vector $\ML$-bundles'' in the literature; see \cite[\S 2]{AnSL}, \cite[\S 3]{PW-BO}\footnote{In \textit{loc.cit.} they are called ``vector $\SLc$-bundles'', but we always replace $\SLc$ with $\ML$.}. However, we prefer to use the above terminology, since it agrees with the corresponding topological one under the real Betti realization. Also, we caution the reader that some authors understand oriented vector bundles as vector $\SL$-bundles; we call these \textit{strictly oriented vector bundles} following \cite[Definition 2.14]{MS_Cellular}.
\end{remark}
\begin{notation}\label{notation_vect_ori} We denote by $\Vect_\ev^\ML(X)$ the groupoid of oriented vector bundles of even rank, i.e., the Zariski sheafification of the groupoid of oriented vector bundles of constant even rank \[ \coprod_n \Vect_{2n}^\ML(X).\]
The notation $\Vect_\ev^\ML\in \PShv(\SmS)$ stands for the corresponding presheaf on smooth $S$-schemes.
\end{notation}

The direct sum of oriented vector bundles induces an $\Einf$-monoid structure on $\Vect_\ev^\ML$; see \cite{Einf-semirings} for a discussion of $\Einf$-monoids/groups/semirings/rings in general (symmetric monoidal) $\infty$-categories with finite products. Moreover, the tensor product endows $\Vect_\ev^\ML$ with the structure of a nonunital (since the unit has rank $1$, which is not even) $\Einf$-semiring; see \cite[\S 5.4.4]{LHA} for a discussion of nonunital objects. There are forgetful morphisms of nonunital $\Einf$-semirings:
\begin{gather}\label{morphisms:stacks_vector_bundles}
    \Vect^{\SL}_\ev\to \Vect_\ev^\ML\to \Vect_\ev,
\end{gather}
where $\Vect^{\SL}_\ev$ (resp. $\Vect_\ev$) denotes the stack of strictly oriented (resp. arbitrary) vector bundles of even rank. These maps correspond to the obvious homomorphisms of linear algebraic groups $\SL_{2n}\to \ML_{2n}\to \GL_{2n}$. 
\begin{remark}
    It is possible to work with the stack of oriented vector bundles of all (not necessarily even) ranks, but it admits only an $\Eone$-monoid structure; see \cite[Example A.0.6]{Mb-modules} for the case of strictly oriented vector bundles. To avoid technicalities involved in working with $\Eone$-monoids, we work with the even-rank substack. 
\end{remark}
\begin{definition}
    The presheaf of \textit{K-theory spaces of oriented vector bundles (of even rank)} is the presheaf of nonunital $\Einf$-rings given by the Zariski localization of the group completion of the stack of oriented vector bundles of even rank: \[\K_\ev^\ML:=\mathrm{L_{Zar}}\Bigl(\Vect^{\ML,\mathrm{gp}}_\ev\Bigr).\] We also denote by $\K_{\mathrm{rk=0}}^\ML=\K^\ML_{\ev}\times_{\underline{2\Z}}\{0\}$ the associated rank-zero presheaf, where $\K_\ev^\ML\to\underline{2\Z}$ is induced by the rank of underlying vector bundles.
\end{definition}
The string of morphisms \eqref{morphisms:stacks_vector_bundles} induces the following morphisms of presheaves of nonunital $\Einf$-rings:
\begin{gather}\label{equat_kring_maps} \K^\SL_\ev\to \K_\ev^\ML\to \K_\ev
\end{gather}
and similarly for the rank zero presheaves, where $\K^\SL_\ev$ (resp. $\K_\ev$) is obtained via a similar procedure from $\Vect^\SL_\ev$ (resp. $\Vect_\ev$). We refer to them as forgetful morphisms.
\section{Motivic Thom spectrum \texorpdfstring{$\mathrm{MML}$}{MML}}\label{section:3}
In this section, we define the motivic Thom spectrum $\MML$ and equip it with a natural $\Einf$-ring structure. Then we construct Thom classes for $\MML$ and obtain an explicit description of $\MML$ in terms of metalinear Grassmannians $\MGr_{n}$.

First, we recall the construction of the motivic Thom functor from \cite[\S 16]{BH-Norms}.
For $X\in \SmS$ denote by $\Pic(\SH(X))$ the $\infty$-groupoid of tensor-invertible objects in $\SH(X)$ and by $\Pic(\SH)$ the corresponding presheaf on $\SmS$. Then the following colimit construction
\[ \mathrm{M}(\beta\colon\mathrm{B}\to \Pic(\SH)):=\colim_{\substack{f\colon X\to S\;\mathrm{smooth} \\ b\in \mathrm{B}(X)}} f_{\#}\beta(b)\]
defines a symmetric monoidal functor of $\infty$-categories $\mathrm{M}\colon\PShv(\SmS)_{/\Pic(\SH)}\to \SHS$. Restricting this functor along the $\Jhom$-homomorphism $\Jhom\colon \K\to \Pic(\SH)$, which is defined as $\mathcal{E}\mapsto \Sigma^\infty\Th(\mathcal{E})$, we obtain a symmetric monoidal functor of $\infty$-categories from the slice category $\PShv(\SmS)_{/\K}$. The motivic Thom spectra $\MSL$ and $\MGL$ are obtained by applying this functor to the rank-zero K-theory presheaves $\KSLrk$ and $\Krk$, see \cite[Theorem 16.13 and Example 16.22]{BH-Norms}. This leads to the following definition.

\begin{definition}
    The \textit{metalinear algebraic cobordism spectrum} is the motivic Thom spectrum associated with the composite $\K_{\mathrm{rk=0}}^\ML\to \Krk\xrightarrow{\iota}\K\colon$ \[\MML_S:=\mathrm{M}(\K^\ML_{\mathrm{rk=0}}\to\K)\in\mathrm{CAlg}(\SHS).\] We usually omit $S$ from the notation and denote this spectrum simply by $\MML$. Applying the motivic Thom functor to the rank zero version of the string of morphisms \eqref{equat_kring_maps}, we obtain the following forgetful morphisms of motivic $\Einf$-ring spectra:
    \[\MSL\to\MML\to\MGL.\]
\end{definition}
\begin{remark}
    Applying the motivic Thom functor to the whole K-theory presheaf of oriented vector bundles $\K^\ML_\ev$, we obtain a $(4,2)$-periodization of $\MML$ that is a motivic $\Einf$-ring spectrum $\mathrm{PMML}$ that admits a decomposition
    \[ \mathrm{PMML}\cong \bigoplus_{n\in\Z} \Sigma^{2n}_{\Proj^1} \MML. \]
    Both $\MML$ and $\mathrm{PMML}$ are not just $\Einf$-ring spectra, but even normed motivic spectra in the sense of Bachmann--Hoyois; see \cite[Example 16.22]{BH-Norms}.
\end{remark}
The above definition can be rephrased in terms of metalinear groups as follows.
\begin{lemma}\label{lemma:motivic_model_tildaK}
    The morphism $\mathrm{B}\ML_\infty = \colim_n \mathrm{B}\ML_{2n} \to \K^\ML_{\mathrm{rk}=0}$, obtained as the colimit over $n$ of the differences between the universal $\ML_{2n}$-torsor and the trivial one, is a motivic equivalence. Moreover, it induces an equivalence of motivic spectra over $S$:
    $$ \colim_{n}\Omega^{2n}_{\Proj^1}\MML_{2n}\xrightarrow{\simeq}\MML, $$
    where $\MML_{2n}$ denotes the Thom spectrum $\mathrm{M}(\mathrm{B}\ML_{2n}\to\mathrm{B}\GL_n\to \K)$.
\end{lemma}
\begin{proof}
    By \cite[Proposition 5.1]{Inf_loop_plus}, the following morphism of presheaves
    $$ \mathrm{B}\ML_\infty\times \underline{2\Z}\to\K_{\ev}^\ML $$
    is an $\A^1$-equivalence on affine schemes, since the cyclic permutation of $\struct^6$ (considered as the direct sum of three copies of the oriented vector bundle $(\struct^2,\struct,\id)$) becomes identity in $\mathrm{L}_{\A^1}(\Vect_{\ev}^\ML)$. In particular, this map is a motivic equivalence. Passing to the rank-zero presheaves, we obtain the first claim. The second one follows since the motivic Thom functor commutes with colimits and inverts motivic equivalences over $\K$ (see \cite[Proposition 16.9(2) and Remark 16.11]{BH-Norms}).
\end{proof}
\begin{lemma}
    The motivic spectrum $\MML$ is stable under base change.
\end{lemma}
\begin{proof}
    This follows from \cite[Proposition 4.10]{BW_Euler}, since each $\ML_{n}$ is an affine, finitely presented group scheme.
\end{proof}

Like any Thom spectrum, $\MML$ comes equipped with the Thom isomorphisms for certain vector bundles, see \cite[Theorem 16.28]{BH-Norms}. We now describe the corresponding Thom classes.
\begin{construction}\label{construction:thom_classes}
    Let $(\mathcal{E},\mathcal{L},\lambda)$ be an oriented vector bundle of rank $n$ over a smooth $S$-(ind)-scheme $X$. By the Yoneda lemma we obtain a morphism of presheaves $X\to \mathrm{B}\ML_{n}$ that classifies $(\mathcal{E},\mathcal{L},\lambda)$. Composing this morphism with $\mathrm{B}\ML_n\to \K^\ML_{\mathrm{rk=0}}$ (if $n$ is odd, this map is the composite $\mathrm{BML}_n\to \mathrm{BML}_{n+1}\to \K^\ML_{\mathrm{rk=0}}$) and applying the motivic Thom functor, we obtain the morphism of motivic spectra
    $$ \Omega^{n}_{\Proj^1}\Th_X(\mathcal{E})\cong\mathrm{M}([\mathcal{E}]-[\struct^n]\colon X\to \K)\to \MML. $$
    The $n$-th $\Proj^1$-suspension of this morphism yields an element in $\MML^{2n,n}(\Th_X(\mathcal{E}))$. 
    We call this element the \textit{Thom class of $(\mathcal{E},\mathcal{L},\lambda)$ in $\MML$-cohomology} and denote it by $\thom(\mathcal{E},\mathcal{L},\lambda)$ (or simply $\thom(\mathcal{E})$ if the orientation is clear from the context).
\end{construction}

\begin{remark}
    In the construction of the Thom class $\thom(\mathcal{E},\lambda)$ in $\MSL$-cohomology for a strictly oriented vector bundle $(\mathcal{E},\lambda)$ over $X$ \cite[\S 5]{PW-Thom}, the authors construct a morphism $X \xrightarrow{} \mathrm{B}\SL_n$ by first mapping to a finite special linear Grassmannians. This requires the use of Jouanolou's device, which is why they impose the condition that every scheme admit an ample family of line bundles. This condition is, in fact, redundant.
\end{remark}

\begin{lemma}
    The above construction endows $\MML$ with a normalized $\ML$-orientation in the sense of \cite[Definition 3.3]{AnSL}. 
\end{lemma}
\begin{proof}
    The functoriality and normalization properties of Thom classes follow from the construction. Hence, it remains to check that the Thom class of the direct sum $\mathcal{E}\oplus \mathcal{E}'$ over $X$ is given by the product of Thom classes of $\mathcal{E}$ and $\mathcal{E}'$ (here we omit orientations for simplicity). Assume that $\mathcal{E}$ (resp. $\mathcal{E}'$) has rank $n$ (resp. $m$) and consider the following commutative diagram:
    \[\xymatrix{ X \ar[r]^-{([\mathcal{E}],[\mathcal{E}'])} & \mathrm{B}\ML_n\times \mathrm{B}\ML_m \ar[r] \ar[d] & \mathrm{B}\ML_{n+m} \ar[d] \\ & \K^\ML_{\mathrm{rk=0}}\times\K^\ML_{\mathrm{rk=0}} \ar[r]^-\mu & \K^\ML_{\mathrm{rk=0}}, }\] where the map $\mu$ comes from the $\Einf$-group structure (it corresponds to the direct sum of oriented vector bundles). Applying the motivic Thom functor, we obtain the result. Indeed, the resulting composition via the top right is the Thom class of $\mathcal{E}\oplus \mathcal{E}'$ (up to $(n+m)$-th $\Proj^1$-suspension), while the composition via the bottom left is the product of the Thom classes (up to the same suspension), as required.
\end{proof}

\begin{remark}\label{remark:univ_prop_MML}
    Let $\EE$ be a motivic homotopy commutative ring spectrum over $S$ equipped with a morphism of homotopy commutative ring spectra $\MML\to \EE$. It follows that such a morphism defines a normalized $\ML$-orientation on $\EE$. Moreover, one can prove a weak universality of $\MML$ among such spectra similar to that established for $\MSL$ in the work of Panin and Walter \cite{PW-Thom}. Weak here means that for every $\ML$-oriented homotopy commutative ring spectrum $\EE$ there exists a map $\MML\to \EE$, but it might be non-unique and non-multiplicative in general (there are potentially non-trivial $\limit^1$-obstruction groups; see \cite[Theorem 5.9]{PW-Thom} for the case of $\MSL$).
\end{remark}

We now describe $\MML$ in terms of metalinear Grassmannians.

\begin{definition}
    The \textit{metalinear Grassmannian} $\MGr_n(\A^m;\Proj^N)\in \SmS$ is the complement of the zero section of the line bundle $\det(\gamma_{n,m}) \boxtimes \struct(2)$ over $\Gr_n(\A^m) \times \Proj^N$. Here $\Gr_n(\A^m)$ denotes the Grassmannian of $n$-dimensional vector subbundles of $\mathcal{O}^m_S$ and $\gamma_{n,m}$ is the tautological rank $n$ vector bundle over $\Gr_n(\A^m)$. 
\end{definition}

Denote by $\gamma^\ML_{n,m,N}$ the pullback of the tautological vector bundle along $\MGr_n(\A^m;\Proj^N)\to \Gr_n(\A^m)$. It admits a canonical structure of an oriented vector bundle, since the 
pullback of $\det(\gamma_{n,m})\boxtimes\struct(2)$ to $\MGr_n(\A^m;\Proj^N)$ has a canonical trivialization.
Denote by $\MGr_n$ the presheaf $\colim_{m,N} \MGr_n(\A^m;\Proj^N)$. We also denote by $\gamma_n^\ML$ the colimit of the oriented vector bundles $\gamma_{n,m,N}^\ML$.

\begin{lemma}\label{lemma:MGr_BML}
    The morphism $\MGr_n\to \mathrm{B}\ML_n$ that classifies $\gamma_n^\ML$ is a motivic equivalence.
\end{lemma}
\begin{proof}
    This is proven in \cite[Corollary 3.7]{BrazeltonWendtMSLc} (see also \cite[proof of Proposition 6.2.1]{Haut}).
\end{proof}

\begin{proposition}
    The maps $[\gamma^\ML_{2n}]-[\struct^{2n}]\colon\MGr_{2n}\to \K^\ML_{\mathrm{rk=0}}$ induce an equivalence of motivic spectra
    \[
    \colim_{n} \Omega^{2n}_{\Proj^1}\Th_{\MGr_{2n}} (\gamma^{\ML}_{2n})\cong\MML.
    \]
\end{proposition}
\begin{proof}
    This follows from Lemma \ref{lemma:MGr_BML} and Lemma \ref{lemma:motivic_model_tildaK}, since the motivic Thom functor inverts motivic equivalences over $\K$ (see \cite[Proposition 16.9(2) and Remark 16.11]{BH-Norms}).
\end{proof}

\section{Interpolation and the main cofiber sequence}\label{section:4}
In this section, we show that the Thom spectrum $\MML$ is equivalent to the free $\MSL$-module associated with $\Omega^1_{\Proj^1}\Th_{\Proj^\infty}(\struct(-2))$, establish the main cofiber sequence and parametrize its splittings. Along the way, we prove several important properties of the K-theory space $\K_\ev^\ML$.
\subsection{Free \texorpdfstring{$\MSL$}{MSL}-module model}
First we recall basics about Quillen's plus construction in $\infty$-topoi (see e.g., \cite{Hoyois_plus}). Let $\mathcal{C}$ be an $\infty$-topos and let $f\colon X\to Y$ be a map in $\mathcal{C}$. Then $f$ is called \textit{acyclic} if it is an epimorphism in the categorical sense. The class of such maps is closed under composition, (co)base change, colimits, and finite products; see \cite[Lemma 3.1]{Inf_loop_plus}. Moreover, if $g\circ f$ and $f$ are acyclic, then $g$ is acyclic. We denote by $X\to X^+$ the final object in the $\infty$-category of acyclic maps out of $X$, and call the respective functor $X\mapsto X^+$ the \textit{plus construction}. We need the following fact.
\begin{lemma}\label{lemma:plus_pullback}
    Suppose that $\mathcal{C}$ is a hypercomplete $\infty$-topos where $\pi_1$ preserves products, and that we have a diagram $X\to Z\leftarrow Y$ in $\mathcal{C}$ such that the objects $Y$ and $Z$ are equivalent to their plus constructions (e.g., $\Eone$-groups). Then the canonical morphism $(X\times_Z Y)^+\to X^+\times_Z Y$ is an equivalence.
\end{lemma}
\begin{proof}
    We have the following commutative diagram in $\mathcal{C}$
    \[\begin{tikzcd}
	 & {X\times_Z Y} & \\
	{(X\times_Z Y)^+} &  & {X^+\times_Z Y.}
	\arrow[from=1-2, to=2-1]
	\arrow[from=1-2, to=2-3]
	\arrow[from=2-1, to=2-3]
    \end{tikzcd}\]
    The left diagonal map is acyclic by definition of the plus construction, and the right diagonal map is acyclic being the base change of $X\to X^+$. It follows that the desired morphism is acyclic as well. By \cite[Corollary 5]{Hoyois_plus}, it remains to show that $\pi_1((X\times_Z Y)^+)$ is hypoabelian (i.e., it does not contain nontrivial perfect subgroups). This holds by the assumption on $\pi_1$, since the map $X\times_Z Y\to (X\times_Z Y)^+$ kills the maximal perfect subgroup of $\pi_1(X\times_Z Y)$; see \cite[Corollary 10]{Hoyois_plus}.
\end{proof}
We will use Quillen's plus construction in the hypercomplete $\infty$-topos $\mathcal{C}=\PShv(\SmS)$.
\begin{notation}
    Consider the presheaf of stable oriented vector bundles of even rank $\sVect_\ev^\ML$ that is defined as the colimit of the sequence \[\Vect_\ev^\ML\xrightarrow{\oplus\, \struct^2}\Vect_\ev^\ML\xrightarrow{\oplus\, \struct^2}\Vect^\ML_\ev\xrightarrow{\oplus\, \struct^2}\cdots\] in $\PShv(\SmS)$. By definition, the canonical map $\Vect_\ev^\ML\to \Vect^{\ML,\mathrm{gp}}_\ev$ factor through $\sVect^\ML_\ev$. We also denote by $\sVect_\ev$ the presheaf of stable vector bundles (of even rank) defined by a similar colimit. The forgetful morphism $\Vect_\ev^\ML\to\Vect_\ev$ induces $\sVect_\ev^\ML\to\sVect_\ev$.
\end{notation}
\begin{lemma}\label{lemma:sVect_ori_plus}
    The morphism of presheaves of nonunital $\Einf$-semirings on $\SmS$ $$\sVect^\ML_\ev\to \Vect^{\mathrm{\ML,\mathrm{gp}}}_\ev$$ is Quillen's plus construction.
\end{lemma}
\begin{proof}
    The statement is a content of the McDuff--Segal group completion theorem \cite{McDuff-Segal} (see also \cite[Corollary 3.4]{Inf_loop_plus}).
\end{proof}
\begin{notation}
Denote by $\sqrt{\det}\colon\Vect^\ML_\ev\to \Pic$ the morphism that maps an oriented vector bundle $(\mathcal{E},\mathcal{L},\lambda)$ to the line bundle $\mathcal{L}$. It is a morphism of presheaves of $\Einf$-monoids, which induces morphisms $\sVect^\ML_\ev\to\Pic$, $\Vect^{\ML,\mathrm{gp}}_\ev\to\Pic$, and $\K^\ML_\ev\to \Pic$.  Slightly abusing notation, we denote all these maps also by $\sqrt{\det}$.
\end{notation}
\begin{proposition}\label{prop:KML_pulback_descr}
    The following square of $\Einf$-groups in $\PShv(\SmS)$ is cartesian:
    \[\begin{tikzcd}
	{\K^\ML_\ev} & {\Pic} \\
	{\K_\ev} & {\Pic.}
	\arrow["\sqrt{\det}", from=1-1, to=1-2]
	\arrow[from=1-1, to=2-1]
	\arrow["\cdot 2", from=1-2, to=2-2]
	\arrow["\det", from=2-1, to=2-2]
\end{tikzcd}\]
\end{proposition}
\begin{proof}
    First, note that by definition, the stack $\Vect^\ML_\ev$ is the pullback of the diagram $$\Vect_\ev\xrightarrow{\det} \Pic \xleftarrow{\cdot 2} \Pic$$ with the maps $\Pic\leftarrow\Vect^\ML_\ev\to\Vect_\ev$ being $\sqrt{\det}$ and the forgetful map, respectively. Then by universality of colimits in the $\infty$-topos of presheaves, we see that the following square is cartesian
    \[\begin{tikzcd}
	{\sVect^\ML_\ev} & {\Pic} \\
	{\sVect_\ev} & {\Pic.}
	\arrow["\sqrt{\det}", from=1-1, to=1-2]
	\arrow[from=1-1, to=2-1]
	\arrow["\cdot 2", from=1-2, to=2-2]
	\arrow["\det", from=2-1, to=2-2]
    \end{tikzcd}\]
    Now using Lemma \ref{lemma:plus_pullback} for $\mathcal{C}=\PShv(\SmS)$ and Lemma \ref{lemma:sVect_ori_plus} (together with a similar equivalence $\sVect^+_\ev\xrightarrow{\simeq}\Vect^{\mathrm{gp}}_\ev$), we obtain that the limit \[\limit(\Vect^{\mathrm{gp}}_\ev\xrightarrow{\det}\Pic\xleftarrow{\cdot 2} \Pic)\] is given by the group completion $\Vect^{\ML,\mathrm{gp}}_\ev$. Moreover, morphisms out of $\Vect^{\ML,\mathrm{gp}}_\ev$ in this pullback square are $\sqrt{\det}$ and the forgetful one, as usual. In order to proceed to the K-theory spaces, it remains to notice that the Zariski localization commutes with finite limits.
\end{proof}
\begin{remark}
    In \cite[Remark 7.11]{HJNY-HermKtheory} and \cite[\S 6.3]{BrazeltonWendtMSLc} (see also \cite[Construction 9.17]{Hoyois_Land} in the non-$\A^1$-invariant context) the authors define the spectrum $\MML$ as the Thom spectrum associated with the rank zero part of \[\limit(\K\xrightarrow{\det} \Pic \xleftarrow{\cdot 2}\Pic).\] This definition is equivalent to ours according to Proposition \ref{prop:KML_pulback_descr}.
\end{remark}
\begin{corollary}\label{cor:KSL_KML_fiber_seq}
    There is a fiber sequence $\K^\SL_\ev\to \K^\ML_\ev\xrightarrow{\sqrt{\det}} \Pic$ in $\PShv(\SmS)$, where the first map is a forgetful morphism. A similar fiber sequence exists on the level of rank zero presheaves. 
\end{corollary}
\begin{proof}
    This follows from Proposition \ref{prop:KML_pulback_descr} and the fiber sequence $\K^\SL_\ev\to \K_\ev\xrightarrow{\det}\Pic$ (see \cite[Example 3.3.4]{Mb-modules}).
\end{proof}
\begin{construction}
    Consider the morphism of groupoids $s(X)\colon\Pic(X) \to \Vect^\ML_\ev(X)$ that maps a line bundle $\mathcal{L}$ over $X \in \SmS$ to $(\mathcal{L}^{\otimes 2}\oplus \struct_X, \mathcal{L}, \id)$. It yields a section of the map $\sqrt{\det}\colon\Vect^\ML_\ev(X) \to \Pic(X)$. After group completion and Zariski localization of the corresponding presheaves, we obtain a section of $\sqrt{\det}\colon\K^\ML_\ev \to \Pic$, which we denote by $s$. Restricting to the rank zero summand, we get a section $s_{\mathrm{rk=0}}\colon\Pic \to \K^\ML_{\mathrm{rk=0}}$ of $\sqrt{\det}_{\mathrm{rk=0}}\colon \K^\ML_{\mathrm{rk=0}}\to\Pic$.
\end{construction}
\begin{proposition}\label{prop:KSLxPic=KML}
    The morphism of presheaves on $\SmS$
    \[\K^{\SL}_\ev \times \Pic \xrightarrow{\id \times s} \K^{\SL}_\ev \times \K^\ML_\ev \xrightarrow{\mathrm{act}} \K^\ML_\ev\]
    is an equivalence.
    Here $\mathrm{act}$ denotes the action on the $\K^{\SL}_\ev$-module $\K^\ML_\ev$. For the rank zero parts of the previous presheaves, this leads to the equivalence $\K^{\SL}_{\mathrm{rk=0}} \times \Pic \cong \K^\ML_{\mathrm{rk=0}}.$
\end{proposition}
\begin{proof}
    This follows from Corollary \ref{cor:KSL_KML_fiber_seq} by the usual topological argument (see \cite[Lemma 3.12]{Splitting_lemma} for a discussion in general $\infty$-categories with finite limits).
\end{proof}
\begin{remark}\label{remark:SGrxP=MGr}
    Combining Proposition \ref{prop:KSLxPic=KML} with the motivic equivalences $\KSLrk\cong_\mot\mathrm{SGr}_\infty$ (here $\SGr_\infty$ is the infinite special linear Grassmannian, see \cite[\S 4]{PW-Thom}), $\Pic\cong_{\mot} \Proj^\infty$, $\mathrm{K^{ML}_{rk=0}}\cong_{\mot}\MGr_\infty$, we obtain that the following composite 
    \begin{equation}\label{equation:SGrxP=MGr}
        \SGr_\infty\times \Proj^\infty\to \SGr_\infty\times \MGr_\infty\to \MGr_\infty
    \end{equation}
    is a motivic equivalence. Here the second map is the action morphism, which actually factors through the multiplication map $\MGr_\infty\times \MGr_\infty\to \MGr_\infty$ on the motivic $\mathrm{H}$-space $\MGr_\infty$. One can also rephrase this at the level of classifying spaces as 
    \[ \mathrm{BSL}_\infty\times \mathrm{B}\Gm\xrightarrow{\simeq}_{\mot}\mathrm{BML}_\infty. \]
\end{remark}
\begin{theorem}\label{thm:mslc_free}
    Let $S$ be a qcqs scheme. The morphism of $\MSL$-modules over $S$
    \[\MSL\wedge \Omega^{1}_{\Proj^1}\Th_{\Proj^\infty}(\struct(-2))\xrightarrow{\id\wedge \Omega^{1}_{\Proj^1}\thom(\struct(-2))} \MSL\wedge \MML\xrightarrow{\mathrm{act}} \MML\] is an equivalence. Here $\mathrm{act}$ is the action on the $\MSL$-module $\MML$, and $\thom(\struct(-2))$ is the Thom class of $\struct(-2)$ in $\MML$-cohomology.
\end{theorem}
\begin{proof}
    It is easy to see that this morphism is a map of $\MSL$-modules. Hence, it suffices to prove that it is an equivalence in $\SHS$. By Proposition \ref{prop:KSLxPic=KML}, the following composition
    \[\K^\SL_{\mathrm{rk=0}}\times \Pic\xrightarrow{\id\times s_{\mathrm{rk=0}}} \K^\SL_{\mathrm{rk=0}}\times \K^\ML_{\mathrm{rk=0}}\xrightarrow{\mathrm{act}} \K^\ML_{\mathrm{rk=0}}\] is an equivalence. Moreover, a straightforward verification shows that this map is an equivalence in $\PShv(\SmS)_{/\K}$ (where $\Pic$ viewed as an object of the slice $\infty$-category via $\Pic\xrightarrow{s_{\mathrm{rk=0}}}\K^\ML_{\mathrm{rk=0}}\to\K$). Therefore, applying the motivic Thom functor, we obtain that the composite
    \[\MSL\wedge \mathrm{M}(\Pic\xrightarrow{s_{\mathrm{rk=0}}}\K^\ML_{\mathrm{rk=0}}\to\K)\xrightarrow{\id\wedge \mathrm{M}(s_{\mathrm{rk=0}})} \MSL\wedge \MML\xrightarrow{\mathrm{act}} \MML\]
    is an isomorphism in $\SHS$.
    It remains to identify the Thom spectrum of $\Pic\to\K$. For this purpose, consider the isomorphism of Thom spectra induced by the usual motivic equivalence $\struct(-1)\colon\Proj^\infty\to\Pic$. It remains to notice that the composition \[\Omega^{1}_{\Proj^1}\Th_{\Proj^\infty}(\struct(-2))\cong \mathrm{M}(\Proj^\infty\xrightarrow{[\struct(-2)]-[\struct]}\K) \xrightarrow{\mathrm{M}(\struct(-1))} \mathrm{M}(\Pic\xrightarrow{s_{\mathrm{rk=0}}}\K^\ML_{\mathrm{rk=0}}\to\K)\xrightarrow{\mathrm{M}(s_{\mathrm{rk=0}})} \MML\] is homotopic to the $\Proj^1$-loops of the Thom class of $\struct(-2)$ in $\MML$-cohomology, see Construction~\ref{construction:thom_classes}.
\end{proof}
\subsection{Proof of the main cofiber sequence}
To compute the cofiber of the morphism $\MSL\to\MML$ and thereby prove Theorem \ref{theorem_A}, we establish several equivalences in the stable $\infty$-category of $\MSL$-modules. We also track the induced maps on $\MGL$-cohomology in order to determine the map $\partial$. 

\begin{lemma}\label{lem:red_to_det}
    Let $X$ be a smooth ind-scheme over $S$ and let $\mathcal{E}$ be a vector bundle of rank $d+1$ over $X$. Then there is an isomorphism of $\MSL$-modules:
    \[ \MSL\wedge \Th_X(\mathcal{E})\xrightarrow{\simeq}\MSL\wedge \Sigma^d_{\Proj^1}\Th_X(\det(\mathcal{E})). \]
\end{lemma}
\begin{proof}
    Consider the formal linear combination $\mathcal{E}-\det(\mathcal{E})-\struct_X^d$ as a map of the form $X\to \Krk$. It admits a lift to a map $X\to \KSLrk$ because of the fiber sequence (see \cite[Example 3.3.4]{Mb-modules})
    \[\KSLrk\to\Krk\xrightarrow{\det}\Pic.\]
    We denote such a lift by $v$. Now let us look at the following triangle
    \begin{center}
        \begin{tikzcd}
            \KSLrk\times X \ar[rr, "\sigma_v"] \ar[rd, "{\mu\circ(\mathrm{can}\times[\mathcal{E}]})"'] & & \KSLrk\times X \ar[ld, "{\mu\circ(\mathrm{can}\times[\det(\mathcal{E})\oplus\struct_X^d])}"] \\ & \K
        \end{tikzcd}
    \end{center}
    Here $\mathrm{can}\colon\KSLrk\to \K$ is the canonical map, $[\mathcal{E}]$ and $[\det(\mathcal{E})\oplus\struct^d_X]$ are classes of the corresponding vector bundles in the K-theory space, the map $\mu\colon\K\times\K\to\K$ comes from the $\Einf$-group structure (it corresponds to the direct sum of vector bundles), and $\sigma_v$ is the shearing map associated with $v$, see \cite[\S 16.5]{BH-Norms}. The diagram commutes up to homotopy, the choice of which is determined by the choice of a path 
    \[[\mathcal{E}]\sim [\mathcal{E}]-[\det(\mathcal{E})]-[\struct^d_X]+[\det(\mathcal{E})\oplus\struct^d_X]\ \text{in}\ \K(X).\]
    Moreover, the morphism $\sigma_v$ is an automorphism of $\KSLrk$-modules since $\KSLrk$ is an $\Einf$-group. Applying the motivic Thom functor, we obtain that the equivalence $\sigma_v$ induces the desired isomorphism of $\MSL$-modules.
\end{proof}

By the usual free-forgetful adjunction
\[ \MSL\wedge-\colon\SHS\leftrightarrows \mathrm{Mod}_{\MSL}(\SHS)\colon \mathrm{Forg} \]
there is an isomorphism $[\MSL\wedge X,\mathrm{E}]_{\MSL}\cong [X,\mathrm{E}]$ for $X\in \SHS$ and $\mathrm{E}\in \mathrm{Mod}_{\MSL}(\SHS)$. In particular, for an $\MSL$-module $\mathrm{E}$ and a map of free $\MSL$-modules $f\colon\MSL\wedge X\to \MSL\wedge Y$, we have the pullback map
\[ f^*\colon[\MSL\wedge Y, \Sigma^{*,*}\EE]_\MSL\to[\MSL\wedge X, \Sigma^{*,*}\EE]_\MSL \]
that corresponds under the above adjunction to a map on $\EE$-cohomology:
\[ f^*\colon \EE^{*,*}(Y)\to \EE^{*,*}(X). \]

\begin{lemma}\label{lem:red_to_det_onMGL}
    Let $X$ be a smooth ind-scheme over $S$ and let $\mathcal{E}$ be a vector bundle of rank $d+1$ over $X$. Then the pullback in $\MGL$-cohomology along the equivalence stated in Lemma \ref{lem:red_to_det} corresponds to the following isomorphism of $\MGL^{*,*}(X)$-modules:
    \begin{align*}
         \MGL^{*-2d,*-d}(\Th_X(\det(\mathcal{E})))&\xrightarrow{\simeq} \MGL^{*,*}(\Th_X(\mathcal{E})) \\
         \alpha\cup\thom(\det(\mathcal{E}))&\mapsto \alpha\cup\thom(\mathcal{E}).
    \end{align*}
\end{lemma}
\begin{proof}
    The morphism of motivic $\Einf$-ring spectra $\MSL\to\MGL$ induces an adjunction (see \cite[Proposition 4.6.2.17]{LHA}):
    \[ \MGL\wedge_{\MSL}-\colon\mathrm{Mod}_{\MSL}(\SH(S))\leftrightarrows \mathrm{Mod}_{\MGL}(\SH(S))\colon\mathrm{Forg}. \]
    It follows that the desired pullback on $\MGL$-cohomology can be computed in the stable $\infty$-category of $\MGL$-modules after applying the left adjoint to the equivalence stated in Lemma \ref{lem:red_to_det}. Consider the following maps of $\MGL$-modules:
    \[ \MGL\wedge\Th_X(\mathcal{E})\xrightarrow{\simeq} \MGL\wedge \Sigma^d_{\Proj^1}\Th_X(\det(\mathcal{E}))\xrightarrow{\simeq} \MGL\wedge \Sigma^{\infty+d+1}_{\Proj^1}X_+,\]
    where the second equivalence is the $\Sigma^d_{\Proj^1}$-suspension of the $\MGL$-Thom isomorphism for $\det(\mathcal{E})$ (in particular, it induces multiplication by $\thom(\det(\mathcal{E}))$ on $\MGL$-cohomology), and we need to compute the pullback along the first equivalence. A straightforward verification shows that the above composite is homotopic to the usual $\MGL$-Thom isomorphism for $\mathcal{E}$, which induces multiplication by $\thom(\mathcal{E})$ on $\MGL$-cohomology. The result follows.
\end{proof}

\begin{lemma}[cf. {\cite[7.13]{DFJK}}]\label{lemma:modulo_squares}
    Let $X$ be a smooth ind-scheme over $S$ and let $\mathcal{L}$ be a line bundle over $X$. Then there is an isomorphism of $\MSL$-modules
    \[ \MSL\wedge \Th_X(\mathcal{L}^{\otimes 2})\xrightarrow{\simeq} \MSL\wedge\Sigma^{\infty+1}_{\Proj^1} X_+. \]
\end{lemma}
\begin{proof}
    Define an equivalence $\MSL\wedge \Sigma^1_{\Proj^1}\Th_X(\mathcal{L}^{\otimes 2})\xrightarrow{\simeq} \MSL\wedge\Sigma^{\infty+2}_{\Proj^1}X_+$ to be the unique (up to homotopy) morphism such that the following diagram commutes
    \begin{equation}\label{cd:modulo_squares}
        \begin{tikzcd}
            \MSL\wedge \Sigma^1_{\Proj^1}\Th_X(\mathcal{L}^{\otimes 2}) \arrow[r, dashed]                                                       & \MSL\wedge\Sigma^{\infty+2}_{\Proj^1}X_+                                                   \\
            \MSL\wedge \Th_X(\mathcal{L}\oplus\mathcal{L}) \arrow[u, "\simeq"] \arrow[r, "\simeq"] & \MSL\wedge \Th_X(\mathcal{L}\oplus\mathcal{L}^\vee), \arrow[u, "\simeq"']
        \end{tikzcd}
    \end{equation}
    where the vertical morphisms are from Lemma \ref{lem:red_to_det} and the bottom horizontal morphism is induced by the non-linear motivic equivalence of Thom spaces $\Th_X(\mathcal{L}\oplus\mathcal{L})\xrightarrow{\simeq}\Th_X(\mathcal{L}\oplus \mathcal{L}^\vee)$ from \cite[Lemma~4.1]{AnSL}. Then the desired equivalence is the $\Proj^1$-desuspension of the equivalence constructed above.
\end{proof}
\begin{lemma}\label{lem:module_squares_onMGL}
    Let $X$ and $\mathcal{L}$ be as above. The element $c_1(\mathcal{L}^\vee)\in\MGL^{2,1}(X)$ is expressed as
    \[ [-1]_F(c_1(\mathcal{L}))=-c_1(\mathcal{L})+a_{11}c_1(\mathcal{L})-(a_{12}+a_{11}^2)c_1(\mathcal{L})^2+\dots, \]
    where $[-1]_F(x)$ is the formal inverse of $x$ with respect to the formal group law of $\MGL$. This element is divisible by $c_1(\mathcal{L})$. More precisely, the element 
    \[ \frac{c_1(\mathcal{L}^\vee)}{c_1(\mathcal{L})}:=\frac{[-1]_F(x)}{x}\Big\vert_{x=c_1(\mathcal{L})}\in \MGL^{0,0}(X) \]
    satisfies the relation $\frac{c_1(\mathcal{L}^\vee)}{c_1(\mathcal{L})}c_1(\mathcal{L})=c_1(\mathcal{L}^\vee)$. 
    
    The pullback on $\MGL$-cohomology along the equivalence stated in Lemma \ref{lemma:modulo_squares} corresponds to the following isomorphism of $\MGL^{*,*}(X)$-modules:
    \begin{align*}
        \MGL^{*-2,*-1}(X) & \xrightarrow{\simeq} \MGL^{*,*}(\Th_X(\mathcal{L}^{\otimes 2})) \\
        \alpha & \mapsto \frac{c_1(\mathcal{L}^\vee)}{c_1(\mathcal{L})}\alpha \cup \thom(\mathcal{L}^{\otimes 2}).
    \end{align*}
\end{lemma}
\begin{proof}
    Consider the universal line bundle $\struct(1)$ over $\Proj^\infty$. The first statement follows for $\mathcal{L}=\struct(1)$ from \[c_1(\struct(1))+_F c_1(\struct(1)^\vee)=c_1(\struct(1)\otimes\struct(1)^\vee)=c_1(\struct)=0,\]
    where $+_F$ denotes the formal sum. The divisibility of $c_1(\struct(1)^\vee)$ follows from the construction of the formal inverse. The general case of these statements follows from the universality of $\Proj^\infty$.
    
    To determine the induced map on $\MGL$-cohomology, consider the commutative diagram
    \begin{center}
        \begin{tikzcd}
            {\MGL^{*-2,*-1}(\Th_X(\mathcal{L}^{\otimes 2}))} \arrow[d, "\simeq"'] & {\MGL^{*-4,*-2}(X)} \arrow[d, "\simeq"] \arrow[l, dashed]                  \\
            {\MGL^{*,*}(\Th_X(\mathcal{L}\oplus\mathcal{L}))}                     & {\MGL^{*,*}(\Th_X(\mathcal{L}\oplus\mathcal{L}^\vee))} \arrow[l, "\simeq"']
        \end{tikzcd}
    \end{center}
    that is induced by \eqref{cd:modulo_squares}. By Lemma \ref{lem:red_to_det_onMGL}, the right vertical map sends an element $\alpha$ to $\alpha\cup\thom(\mathcal{L}\oplus\mathcal{L}^\vee)$ and the left vertical map sends $\beta\cup \thom(\mathcal{L}^{\otimes 2})$ to $\beta\cup \thom(\mathcal{L}\oplus \mathcal{L})$. Hence, it remains to determine the bottom map, which we denote by $f^*$ (this is the pullback along $f\colon\Th_X(\mathcal{L}\oplus\mathcal{L})\xrightarrow{\simeq}\Th_X(\mathcal{L}\oplus\mathcal{L}^\vee)$). 
    For $\gamma\in\MGL^{*,*}(X)$ denote by $g(\gamma)$ the unique element of $\MGL^{*,*}(X)$ that satisfies 
    \[ f^*(\gamma\cup\thom(\mathcal{L}\oplus\mathcal{L}^\vee))=g(\gamma)\cup\thom(\mathcal{L}\oplus\mathcal{L}). \]
    We claim that $g(\gamma)=\frac{c_1(\mathcal{L}^\vee)}{c_1(\mathcal{L})}\gamma$. It suffices to prove the claim for $\gamma=1$ since $f^*$ is an isomorphism of $\MGL^{*,*}(X)$-modules.
    Consider the commutative diagram
    \begin{center}
        \begin{tikzcd}
            {\MGL^{*,*}(\Th_X(\mathcal{L}\oplus \mathcal{L}))} &                                       & {\MGL^{*,*}(\Th_X(\mathcal{L}\oplus \mathcal{L}^\vee))} \arrow[ll, "\simeq"'] \\
                        & {\MGL^{*,*}(X),} \arrow[ru] \arrow[lu] &                                                   
        \end{tikzcd}
    \end{center}
    where the diagonal maps are induced by zero sections. It gives the following equality
    \[ c_1(\mathcal{L})^2 g(1)=c_1(\mathcal{L}) c_1(\mathcal{L}^\vee) \]
    and it remains to divide both parts by $c_1(\mathcal{L})^2$. For that we first obtain the claim for  $X=\Proj^\infty$ and $\mathcal{L}=\struct(1)$, and then conclude by universality. This is done in order to get around the problem that $c_1(\mathcal{L})$ is zero divisor in general.
\end{proof}
For the sake of completeness, we also recover the following more general fact.
\begin{proposition}
    Let $X$ be a smooth ind-scheme over $S$ and let $(\mathcal{E},\mathcal{L},\lambda)$ be an oriented vector bundle of rank $d$ over $X$. Then there is an isomorphism of $\MSL$-modules
    \[\MSL\wedge \Th_X(\mathcal{E})\cong \MSL\wedge\Sigma^{\infty+d}_{\Proj^1} X_+. \]
\end{proposition}
\begin{proof}
    Applying Lemma \ref{lem:red_to_det}, we can replace $\mathcal{E}$ by $\det(\mathcal{E})\cong \mathcal{L}^{\otimes 2}$. Then the result follows from Lemma~\ref{lemma:modulo_squares}.
\end{proof}
\begin{remark}\label{ques:twisted_thom_mult}
    Note, however, that these Thom isomorphisms are not multiplicative. More precisely, assume that $(\mathcal{E}_1,\mathcal{L}_1,\lambda_1)$ and $(\mathcal{E}_2,\mathcal{L}_2,\lambda_2)$ are oriented vector bundles over $X\in \SmS$ and denote by $t(\mathcal{E}_i)\in\MSL^{2d_i,d_i}(\Th_X(\mathcal{E}_i))$ the ``Thom class'' that is induced by the constructed Thom isomorphism. Then we have
    \[ t(\mathcal{E}_1\oplus \mathcal{E}_2)\neq t(\mathcal{E}_1)\cup t(\mathcal{E}_2).\]
    Indeed, if the equality holds, then replacing $\MSL$ with $\MGL$ via the map $\MSL\to\MGL$, we obtain a similar formula in $\MGL$-cohomology. However, one can check using Lemma \ref{lem:red_to_det_onMGL} and Lemma \ref{lem:module_squares_onMGL} that the following formula is valid 
    \[  t^\MGL(\mathcal{E})=\frac{[-1]_F(c_1(\det(\mathcal{E})))}{c_1(\det(\mathcal{E}))}\cup \thom(\mathcal{E}) \]
    for an oriented vector bundle $\mathcal{E}$.
    Here $t^\MGL(\mathcal{E})$ is the image of $t(\mathcal{E})$ in $\MGL$-cohomology and $\thom(\mathcal{E})$ is the Thom class of $\mathcal{E}$ for the canonical orientation of $\MGL$.
    Now it is easy to see that this additional multiplier is not multiplicative. This provides a negative answer to Ananyevskiy's question \cite[Remark 4.4]{AnSL} (the isomorphism he constructed is slightly different from ours, but the same multiplier appears in his context as well because of the equivalence from \cite[Lemma 4.1]{AnSL}). We also remark that by Haution's result this problem disappears after $\eta$-periodization \cite[Proposition 3.2.8]{Haut}.
\end{remark}
The equivalences obtained above are enough to deduce the split cofiber sequence from Theorem \ref{theorem_A} without determination of the operation $\partial$. To do so, we need to control the group of operations (of $\MSL$-modules) of the form $\MML\to\Sigma^1_{\Proj^1}\MGL$. This requires two more lemmas.
\begin{lemma}\label{lem:MGL(MGr)}
    There are isomorphisms of $\MGL^{*,*}(S)$-algebras:
    \begin{align*}
        \MGL^{*,*}(\MGr_\infty)&\cong\limit_n\MGL^{*,*}(\MGr_n) \\ &\cong \limit_n\biggl(\bigslant{\MGL^{*,*}(S)[\![c_1(\sqrt{\det}),c_1,c_2,\dots,c_n]\!]}{\Bigl(c_1(\det)-[2]_F(c_1(\sqrt{\det}))\Bigr)}\biggr).
    \end{align*}
    Here $c_i$ is the $i$-th Chern class $c_i(\gamma_n^\ML)$, $c_1(\sqrt{\det})$ is the first Chern class of the line bundle $\sqrt{\det\gamma^\ML_n}$ (which is the pullback of $\struct(-1)$ along the obvious projection $\MGr_n\to\Proj^\infty$), $[2]_F(c_1(\sqrt{\det}))$ denotes the formal multiplication of $c_1(\sqrt{\det})$ by $2$ with respect to the formal group law of $\MGL$, and $c_1(\det)$ is the shorthand for $c_1(\det\gamma_n^\ML)$ (which is a power series in the $c_i$).
\end{lemma}
\begin{proof}
    First, we compute $\MGL^{*,*}(\MGr_n)$. Consider the cofiber sequence of motivic spectra over $S$:
    \begin{equation}\label{equation:cofib_MGr_GrP}
        \Sigma^\infty_{\Proj^1}(\MGr_{n})_+\to \Sigma^\infty_{\Proj^1}(\Gr_n\times\Proj^\infty)_+\to \Th(\det(\gamma_n)\boxtimes\struct(2)), 
    \end{equation}
    which comes from the definition of the Thom space.
    It induces the following long exact sequence of $\MGL$-cohomology groups:
    \[ \dots\to \MGL^{*,*}(\Th(\det(\gamma_n)\boxtimes\struct(2)))\xrightarrow{(1)} \MGL^{*,*}(\Gr_n\times\Proj^\infty)\to \MGL^{*,*}(\MGr_n)\to \cdots. \]
    Since the second morphism in \eqref{equation:cofib_MGr_GrP} is induced by the zero section, the homomorphism (1) sends the Thom class $\thom(\det(\gamma_n)\boxtimes\struct(2))$ to the first Chern class $c_1(\det(\gamma_n)\boxtimes\struct(2))$. Therefore, the $\MGL^{*,*}(S)$-algebra $\MGL^{*,*}(\MGr_n)$ is the quotient of (see \cite[Proposition 6.2(1)]{NSOLandw})
    \[ \MGL^{*,*}(\Gr_n\times \Proj^\infty)\cong \MGL^{*,*}(S)[\![h,c_1,\dots,c_n]\!] \]
    by the principal ideal generated by $c_1(\det(\gamma_n)\boxtimes\struct(2))=c_1(\det\gamma_n)+_F[-2]_F(h)$, where $h=c_1(\struct(-1))$. Here the element $c_i$ corresponds to the $i$-th Chern class $c_i(\gamma_n^\ML)$ and $h$ corresponds to the first Chern class of the pullback of $\struct(-1)$ along $\MGr_n\to \Proj^\infty$. Because of that, we adopt the notation from the statement of the lemma. 
    We claim that the principal ideal generated by $c_1(\det)+_F[-2]_F(c_1(\sqrt{\det}))$ coincides with the principal ideal generated by $c_1(\det)-[2]_F(c_1(\sqrt{\det}))$. This holds because the vanishing of both elements in $\MGL^{*,*}(\MGr_n)$ is a consequence of the isomorphism of line bundles $\det(\gamma_n^\ML)\cong\struct(-2)$ (more precisely, any of them is zero in the quotient by the principal ideal generated by another one).

    Taking the limit over $n$, we obtain the second isomorphism. Applying the Milnor exact sequence, we obtain the first isomorphism since all homomorphisms in the inverse system are surjective, namely quotients by the top Chern classes.
\end{proof}
\begin{lemma}\label{lem:MGL(MML)}
    The bigraded group of $\MSL$-linear operations $[\MML,\Sigma^{*,*}\MGL]_{\MSL}$ maps injectively into the bigraded group of all operations $[\MML,\Sigma^{*,*}\MGL]$. Moreover, this inclusion can be identified with the map
    \[ \MGL^{*,*}(\Proj^\infty)\to \MGL^{*,*}(\MGr_\infty) \]
    of $\MGL^{*,*}(S)$-modules that sends a power series in the first Chern class $f(c_1(\struct(-1)))$ to $f(c_1(\sqrt{\det}))$.
\end{lemma}
\begin{proof}
    It suffices to prove the second statement since injectivity then will follow from Lemma \ref{lem:MGL(MGr)} and the projective bundle formula for $\MGL$-cohomology. Theorem \ref{thm:mslc_free} implies that the bigraded group of $\MSL$-linear operations $[\MML,\Sigma^{*,*}\MGL]_\MSL$ is isomorphic to $\MGL^{*+2,*+1}(\Th_{\Proj^\infty}(\struct(-2)))$. Applying the Thom isomorphisms, we identify the desired homomorphism with a map of the form
    \[ \MGL^{*,*}(\Proj^\infty)\to \MGL^{*,*}(\MGr_\infty) \]
    and it remains to compute the image of $f(c_1(\struct(-1)))=:\alpha$. A straightforward verification shows that this map sends an element $\alpha$ of bidegree $(p,q)$ to the composition
    \[ \Sigma^\infty_{\Proj^1}(\MGr_\infty)_+\xleftarrow{\simeq} \Sigma^\infty_{\Proj^1}(\SGr_\infty\times\Proj^\infty)_+\xrightarrow{\mathrm{pr}_2^*(\alpha)} \Sigma^{p,q}\MGL, \]
    where the first isomorphism is induced by the motivic equivalence $\SGr_\infty\times\Proj^\infty\xrightarrow{\simeq}_{\mot}\MGr_\infty$ (see Remark \ref{remark:SGrxP=MGr}). Cohomology of the product $\SGr_\infty\times \Proj^\infty$ is given by the tensor product of cohomology rings (over $\MGL^{*,*}(S)$) by the projective bundle formula. It remains to notice that the pullback of $c_1(\sqrt{\det})$ along $\SGr_\infty\times \Proj^\infty\to \MGr_\infty$ is given by $c_1(\struct(-1))$ by the construction of this map, functoriality of Chern classes, and $c_1(\det\gamma^{\SL}_\infty)=0$.
\end{proof}

Now we are ready to prove the main result of this section.
\begin{theorem}\label{thm:main_cofib_seq}
    The forgetful morphism $\MSL\to \MML$ is included into the following (split) cofiber sequence
    $$ \MSL\to \MML\xrightarrow{\partial} \Sigma^1_{\Proj^1} \MGL\xrightarrow{0} \Sigma\MSL, $$
    of $\MSL$-modules over $S$, where $\partial$ is the unique $\MSL$-linear operation that corresponds to the characteristic class $[-1]_F(c_1(\sqrt{\det}))$ under the Thom isomorphism, see Lemma \ref{lem:MGL(MML)}.
    The set of homotopy classes of splittings of this cofiber sequence (in the category of $\MSL$-modules) is in bijection with
    $$ \{ t\in \MSL^{2,1}(\Th_{\Proj^\infty}(\struct(-2)))\,|\,t\vert_{\Proj^0}=\Sigma^1_{\Proj^1} 1 \}. $$
    Here the restriction is taken along $p\colon S=\Proj^0\to \Proj^\infty$ given by the point $(1:0:\cdots)$ of $\Proj^\infty$.
\end{theorem}
\begin{proof}
    Consider the following commutative diagram in which rows are cofiber sequences of $\MSL$-modules:
    \begin{center}
        \begin{tikzcd}
            \MSL \arrow[d, equal] \arrow[r] & \MML \arrow[r]                                                                                                 & \cofib_1                     \\
            \MSL \arrow[d, equal] \arrow[r] & \MSL\wedge\Omega^1_{\Proj^1}\Th_{\Proj^\infty}(\struct(-2)) \arrow[r] \arrow[u, "\simeq"', "(1)"] \arrow[d, "(2)"', "\simeq"] & \cofib_2 \arrow[u] \arrow[d] \\
            \MSL \arrow[d, equal] \arrow[r] & \MSL\wedge\Sigma^\infty_{\Proj^1}\Proj^\infty_+ \arrow[r] \arrow[d, "(3)"', "\simeq"]                                 & \MSL\wedge\Sigma^\infty_{\Proj^1}(\Proj^\infty,p) \arrow[d]           \\
            \MSL \arrow[r, "\mathrm{inc}_1"]              & \MSL\oplus\Sigma^{1}_{\Proj^1}\MGL \arrow[r, "\mathrm{pr}_2"]                                                  & \Sigma^{1}_{\Proj^1}\MGL.   
        \end{tikzcd}
    \end{center}
    Here morphism (1) is from Theorem \ref{thm:mslc_free}, morphism (2) is from Lemma \ref{lemma:modulo_squares}, and morphism (3) has components $\MSL\wedge\Sigma^{\infty}_{\Proj^1}(\Proj^\infty\to S)_+$ and $\MSL\wedge\Sigma^{\infty}_{\Proj^1}(\Proj^\infty,p)\cong \Sigma^1_{\Proj^1}\MGL$ from \cite[Corollary 3.3]{Egor} (see also \cite[Theorem 1.1]{Ahina}). All other morphisms are the obvious ones. The diagram implies that $\cofib_1\cong \Sigma^{1}_{\Proj^1}\MGL$ and the boundary morphism in the first cofiber sequence is zero. Hence, it remains to compute the map $\MML\to \Sigma^{1}_{\Proj^1}\MGL$ and check the last statement.

    First, we compute the second map in the cofiber sequence. The above diagram shows that it is given by the following composition:
    \begin{align*}
        \MML\xleftarrow{\simeq} \MSL\wedge \Omega^1_{\Proj^1}\Th_{\Proj^\infty}(\struct(-2)) \xrightarrow{\simeq} \MSL\wedge\Sigma^\infty_{\Proj^1}\Proj^\infty_+\to \MSL\wedge\Sigma^\infty_{\Proj^1}(\Proj^\infty,p)\xrightarrow{\simeq}\Sigma^1_{\Proj^1}\MGL.
    \end{align*}
    The last equivalence corresponds to $c_1(\struct(-1))\in \MGL^{2,1}(\Proj^\infty,p)$ by the construction of this equivalence. The composition of third and fourth morphisms corresponds to the same element since the pullback along the third map on $\MGL$-cohomology is the inclusion of $\MGL^{*,*}(\Proj^\infty,p)$ into $\MGL^{*,*}(\Proj^\infty)$. Applying Lemma \ref{lem:module_squares_onMGL}, we obtain that the composition of the last three morphisms corresponds to the element
    \[ \frac{c_1(\struct(1))}{c_1(\struct(-1))}c_1(\struct(-1))\cup \thom(\struct(-2))=c_1(\struct(1))\cup \thom(\struct(-2)). \]
    The first Chern class $c_1(\struct(1))$ is expressed as $[-1]_F(c_1(\struct(-1)))$ by Lemma \ref{lem:module_squares_onMGL} again. Applying Lemma \ref{lem:MGL(MML)}, we obtain that the whole composition is given by a unique $\MSL$-linear map that under the forgetful functor and Thom isomorphism corresponds to $[-1]_F(c_1(\sqrt{\det}))$, as required.

    Now we prove the last statement. The set of homotopy splittings of the first cofiber sequence is in bijection with the set of homotopy retractions of the morphism $\MSL\to \MML$. By the left upper commutative square in the diagram above, this set is in bijection with the set of (homotopy) retractions of $\MSL\to \MSL\wedge \Omega^1_{\Proj^1}\Th_{\Proj^\infty}(\struct(-2))$, which is induced by $p\colon \Proj^0\hookrightarrow \Proj^\infty$. Applying the free-forgetful adjunction, we obtain that the desired set is in bijection with
    \[ \{ t\in\MSL^{0,0}(\Omega^1_{\Proj^1}\Th_{\Proj^\infty}(\struct(-2)))\, |\, (\Omega^1_{\Proj^1}\Th(p))^*(t)=1\}. \]
    Here $\Th(p)$ denotes the map on Thom spaces $\Th_S(\struct)\to \Th_{\Proj^\infty}(\struct(-2))$ induced by $p$, and the equality means that the pullback of $t$ corresponds to $1\in\MSL^{0,0}(S)$ under the equivalence $\Omega^1_{\Proj^1}\Th_S(\struct)\cong \Omega^1_{\Proj^1}\Sigma^1_{\Proj^1}\sph\cong \sph$.
    This datum is equivalent to the desired one via the isomorphism
    \[ \MSL^{0,0}(\Omega^1_{\Proj^1}\Th_{\Proj^\infty}(\struct(-2))) \cong \MSL^{2,1}(\Th_{\Proj^\infty}(\struct(-2))), \]
    which is induced by the autoequivalence $\Sigma^1_{\Proj^1}$. This concludes the proof.
\end{proof}
\begin{remark}\label{rem:mult_splitting}
    A natural question arising from Theorem \ref{thm:main_cofib_seq} is whether there exists a multiplicative splitting of our cofiber sequence. The answer is already negative for the corresponding commutative algebras in the homotopy category of $\MSL$-modules. More precisely, we claim that $\MSL$ is not a retract of $\MML$ in the category of homotopy commutative $\MSL$-algebras\footnote{Here by a homotopy commutative $\MSL$-algebra we mean a commutative monoid in the homotopy category of $\MSL$-modules.}. In fact, we prove more: the morphism $\MSL\to \MML$ does not admit a retraction in the category of homotopy commutative ring spectra (i.e., the category of commutative monoids in the homotopy category $\mathrm{Ho}(\SHS)$). To prove this, we use computations of homotopy groups from \S \ref{section_5.2} and Appendix~\ref{appendix_char2}. 

    Assume that there is a retraction of $\MSL\to\MML$ in the category of homotopy commutative ring spectra over $S$. Taking the base change to the algebraic closure of the residue field of a point $s\in S$, we obtain the corresponding retraction over an algebraically closed field of exponential characteristic $e$. Applying $\pi_{2*,*}(-)[\nicefrac{1}{2e}]$ (the inversion of $e$ is omitted from the notation below) and using Theorem \ref{thm:pullback_descr} and \cite[Theorem D]{Egor} (or Theorem \ref{theorem:geom_diags_char2} if the characteristic of $\kappa(s)$ is $2$), we obtain a ring retraction of 
    \[ \pi_{2*}(\MSU)[\nicefrac{1}{2}]\to \pi_{2*}(\mathrm{M}\Sigma)[\nicefrac{1}{2}]\cong \pi_{2*}(\MU)[\nicefrac{1}{2}], \]
    see Corollaries \ref{corollary:geom_diag_away_2} and \ref{corollary:MML_MGL_char2} for the last isomorphism. However, this map does not admit a retraction in the category of graded commutative rings. The ring $\pi_{2*}(\MSU)[\nicefrac{1}{2}]$ is isomorphic to $\Z[\nicefrac{1}{2}][y_2,y_3,\dots]$ with $\mathrm{deg}(y_i)=2i$ (see \cite[Theorem 7.1]{CLP}), while the ring $\pi_{2*}(\mathrm{M}\Sigma)[\nicefrac{1}{2}]$ is isomorphic to $\mathbb{L}[\nicefrac{1}{2}]\cong\Z[\nicefrac{1}{2}][a_1,a_2,\dots]$ with $\mathrm{deg}(a_i)=2i$ (see \cite{Quillen}). If the characteristic of $\kappa(s)$ is not $3$, the absence of a ring retraction follows from the fact that $y_3$ maps to $-3a_3+2a_1a_2$ by \cite[Lemma~4.4]{MSL-slices} (the generator $y_3$ coincides with $x_3$ by the discussion just below \cite[Theorem~7.1]{CLP}). In the remaining case, one can use that the image of $y_5$ equals to $\pm 5a_5$ modulo decomposable elements (this follows from the fact that the polynomial generators $a_i$ and $y_i$ are determined by their Milnor numbers: $s_5(y_5)=\pm 5$ and $s_5(a_5)=\pm 1$, see \cite[Theorems 1.5 and 2.1]{CLP}).
\end{remark}

As a corollary, we generalize an equivalence proved by Brazelton and Wendt over fields \cite[Corollary 1.3]{BrazeltonWendtMSLc} (see also Haution's result \cite[Proposition 3.2.8]{Haut}).

\begin{corollary}\label{cor_eta_period_mslc}
    The map $\MSL\to\MML$ induces an equivalence of motivic $\Einf$-ring spectra over $S$
    \[\MSL[\eta^{-1}]\xrightarrow{\simeq} \MML[\eta^{-1}].\]
\end{corollary}
\begin{proof}
    Smashing the cofiber sequence from Theorem \ref{thm:main_cofib_seq} with $\sph[\eta^{-1}]$, we obtain the cofiber sequence $$\MSL[\eta^{-1}]\to \MML[\eta^{-1}]\to \Sigma^{1}_{\Proj^1}\MGL[\eta^{-1}]=0,$$
    where the vanishing follows from the fact that $\eta$ is trivial in $\MGL$ (see \cite[Theorem 3.8]{Hoy15}). Hence, the first map is an isomorphism.
\end{proof}
\section{Computations of homotopy groups}\label{section:5}
In this section, we apply the split cofiber sequence obtained in \S\ref{section:4}  to compute certain homotopy groups of $\MML$ over a field. We deduce formulas for the first Milnor--Witt stems of $\MML$ and compute its geometric diagonal (away from the exponential characteristic).
\subsection{First Milnor--Witt stems}
In this subsection we work over a field $F$. 
We begin with the following proposition concerning the negative and zeroth lines. For the definition of the Milnor--Witt K-theory see \cite[Definition 1.21]{Morel_book}.
\begin{proposition}\label{prop:pi_0_MML}
    We have $\pi_{n+m,n}(\MML)=0$ for $m<0$. The unit map $\sph\to\MML$ induces an isomorphism \[ \bigoplus_{n\in\Z}\pi_{n,n}(\sph)\cong \bigoplus_{n\in\Z}\pi_{n,n}(\MML)\]
    of graded rings. The left hand side can be identified with the Milnor--Witt K-theory $\mathrm{K}^{\mathrm{MW}}_{-*}(F)$ by Morel's computation; see \cite[Theorem 1.23]{Morel_book}.
\end{proposition}
\begin{proof}
    The vanishing of the negative lines follows from the connectivity of $\MML$ with respect to the homotopy $t$-structure, which is standard; see \cite[Lemma 3.1]{Hoy15}. The isomorphism for the zeroth line follows from our main cofiber sequence, connectivity of $\MGL$, and the analogous result for $\MSL$; see \cite[Example 16.35]{BH-Norms}.
\end{proof}
We now turn to the first, second, and third lines. Since these computations use the slices of the algebraic cobordism spectrum, we need to invert the exponential characteristic. 
\begin{notation}
    Following \cite{ARO}, we denote by $\mathrm{kq}\in\mathrm{CAlg}(\SH(F))$ the very effective cover of the hermitian K-theory spectrum (see \cite[\S 4]{KRO_hermitian_k_theory_char2} for the case $\operatorname{char}(F)=2$). There is a canonical morphism $\MML\to\mathrm{kq}$ of motivic $\Einf$-ring spectra (see \cite[Remark 7.11 and Remark 7.3]{HJNY-HermKtheory}). We also denote by $\mathrm{kgl}\in\mathrm{CAlg}(\SH(F))$ the (very) effective cover of the algebraic K-theory spectrum, and by $\tilde{\partial}\colon \MML\to \Sigma^{1}_{\Proj^1}\mathrm{kgl}$ the composite of $\partial$ with the $\Proj^1$-suspension of the canonical $\Einf$-ring map $\MGL\to \mathrm{kgl}$ (see \cite[\S 6]{Hilb_K_theory}).
\end{notation}
\begin{proposition}\label{prop:pi_1_MML}
    Assume that the field $F$ has characteristic zero. Then the morphisms $\MML\to \mathrm{kq}$ and $\tilde{\partial}\colon\MML\to \Sigma^{1}_{\Proj^1}\mathrm{kgl}$ induce an isomorphism \[\pi_{n+1,n}(\MML)\cong\pi_{n+1,n}(\mathrm{kq})\oplus \mathrm{K}^{\mathrm{M}}_{1-n}(F), \]
    where $\mathrm{K}^{\mathrm{M}}_{*}(F)$ denotes the Milnor K-theory. If $F$ has characteristic $p>0$, the same result holds after inverting $p$.
\end{proposition}
\begin{proof}
    Below we implicitly invert the exponential characteristic. The cofiber sequence from Theorem \ref{thm:main_cofib_seq} yields the following short exact sequence
    \[ 0\to\pi_{n+1,n}(\MSL)\to \pi_{n+1,n}(\MML)\xrightarrow{\partial_*} \pi_{n-1,n-1}(\MGL) \to 0. \]
    The group on the right is isomorphic to $\mathrm{K}^{\mathrm{M}}_{1-n}(F)$ by \cite[Remark 3.10]{Hoy15}. The exact sequence splits since the composite $\MSL\to \MML\to \mathrm{kq}$ induces an isomorphism $\pi_{n+1,n}(\MSL)\cong\pi_{n+1,n}(\mathrm{kq})$ by \cite[Theorem 1.4]{Ahina} (see \cite[Theorem B(1)]{MSL-slices} in characteristic $2$). Therefore we obtain the decomposition
    \[\pi_{n+1,n}(\MML)\cong \pi_{n+1,n}(\mathrm{kq})\oplus \mathrm{K}^{\mathrm{M}}_{1-n}(F),\]
    induced by the canonical map $\MML\to \mathrm{kq}$ and $\partial$. It remains to observe that the $\Einf$-ring map $\MGL\to \mathrm{kgl}$ induces an isomorphism on $\pi_{n-1,n-1}$. This follows, even without inversion of the exponential characteristic, from the cofiber sequence $\Sigma^{1}_{\Proj^1}\mathrm{kgl}\xrightarrow{\beta}\mathrm{kgl}\to \mathrm{M}\Z$ (see \cite[Proposition 2.7]{ARO}) 
    and the fact that the morphism $\MGL\to \mathrm{M}\Z$ induces an isomorphism on $\pi_{n-1,n-1}$, see \cite[Lemma~7.5]{Hoy15}.
\end{proof}
\begin{remark}
    One can show that the $\K^{\mathrm{MW}}_{-*}(F)$-submodule $\K^\mathrm{M}_{1-*}(F)\subset \pi_{*+1,*}(\MML)$ is generated by the class of the projective line $[\Proj^1,\theta]\in\pi_{2,1}(\MML)$ for an obvious orientation $\theta$ of $\det(T_{\Proj^1})\cong \struct(2)$. The construction of the class of an oriented smooth projective variety here is analogous to the construction of the class of a Calabi--Yau variety (or strictly oriented smooth projective variety, following our terminology) in the geometric diagonal of $\MSL$, see \cite[Definition 3.2]{LYZ}.
\end{remark}
\begin{notation}
    We denote by $H^{s,w}$ the motivic cohomology group $H^s(F,\Z(w))$. We also use standard notations for operations on motivic cohomology groups such as the projection on mod-2 motivic cohomology $\mathrm{pr}^\infty_2$, the Steenrod squares $\mathrm{Sq}^i$, and so on.
\end{notation}
\begin{proposition}\label{prop:pi_2_MML}
    Assume that the field $F$ has characteristic zero. Then the morphisms $\MML\to \mathrm{kq}$ and $\tilde\partial\colon\MML\to \Sigma^{1}_{\Proj^1}\MGL$ induce an isomorphism \[\pi_{n+2,n}(\MML)\cong\pi_{n+2,n}(\mathrm{kq})\oplus \pi_{n,n-1}(\mathrm{kgl}), \]
    where the group $\pi_{n,n-1}(\mathrm{kgl})$ fits into the following exact sequence
    \[0\to\bigslant{\mathrm{K}^{\mathrm{M}}_{2-n}(F)}{\Bigl(\partial^2_\infty\mathrm{Sq}^2\mathrm{pr^\infty_2}H^{-1-n,1-n}\Bigr)}\to \pi_{n,n-1}(\mathrm{kgl})\to H^{-n,1-n}\to 0.\] If $F$ has characteristic $p>0$, the same result holds after inverting $p$.
\end{proposition}
\begin{proof}
    We omit inversion of the exponential characteristic below. Similarly to the proof of Proposition \ref{prop:pi_1_MML}, one can deduce $\pi_{n+2,n}(\MML)\cong \pi_{n+2,n}(\mathrm{kq})\oplus \pi_{n,n-1}(\MGL)$ from the fact that the $\Einf$-ring map $\MSL\to \mathrm{kq}$ induces an isomorphism $\pi_{n+2,n}(\MSL)\cong \pi_{n+2,n}(\mathrm{kq})$ (see \cite[Theorem B(1)]{MSL-slices}). To prove that the map $\MGL\to \mathrm{kgl}$ induces an isomorphism on $\pi_{n,n-1}$, consider the slice spectral sequences for $\MGL$ and $\mathrm{kgl}$ (see the beginning of \S\ref{subsection:6-1} and the references therein). The slices of $\MGL$ and $\mathrm{kgl}$ are computed in \cite[\S 8.3]{Hoy15} and \cite{Voev_slices, Levine_coniveau} respectively. The morphism $\slice_q(\MGL)\to \slice_q(\mathrm{kgl})$ can be identified with the map $\Sigma^{q}_{\Proj^1}\mathrm{M}\mathbb{L}_q\to\Sigma^{q}_{\Proj^1}\mathrm{M}\Z\{\beta^q\}$ induced by the ring homomorphism $\mathbb{L}\to \Z[\beta]$ classifying the multiplicative formal group law. Here, $\mathbb{L}_q$ is the degree $q$ part of the Lazard ring and $\beta^q$ is the $q$-th power of the Bott periodicity element. In particular, $\MGL\to \mathrm{kgl}$ induces an isomorphism on $\slice_q$ for $q\leq 1$. The first slice differential for $\MGL$ is determined in \cite[Theorem 4.2]{MSL-slices}. It follows that the only nontrivial entries on the $E^2$-page for $\MGL$ (resp. $\mathrm{kgl}$) contributing to $\pi_{n,n-1}(\MGL)$ (resp. $\pi_{n,n-1}(\mathrm{kgl})$) are the right and the left hand groups of the desired extension. There is no room for higher differentials and the result follows from the strong convergence of the slice spectral sequences for $\MGL$ and $\mathrm{kgl}$.
\end{proof}
In turn, the third line is not covered by the corresponding lines of $\mathrm{kq}$ and $\Sigma^1_{\Proj^1}\mathrm{kgl}$:
\begin{proposition}\label{prop:pi_3_MML}
    Assume that the field $F$ has characteristic zero. Then the morphisms $\MML\to \mathrm{kq}$ and $\tilde\partial\colon\MML\to \Sigma^{1}_{\Proj^1}\mathrm{kgl}$ induce a short exact sequence
    \[ \bigslant{\mathrm{K}^\mathrm{M}_{3-n}(F)}{\Bigl(\partial^2_\infty\mathrm{Sq}^6\mathrm{pr}^\infty_2H^{-n-4,-n}\Bigr)}\oplus\mathrm{K}^\mathrm{M}_{3-n}(F)\hookrightarrow \pi_{n+3,n}(\MML)\twoheadrightarrow\pi_{n+3,n}(\mathrm{kq})\oplus \pi_{n+1,n-1}(\mathrm{kgl}).\]
    If $F$ has characteristic $p>0$, the same result holds after inverting $p$.
\end{proposition}
\begin{proof}
    Below we implicitly invert the exponential characteristic of $F$. Consider the following diagram:
    \[ \xymatrix{0 \ar[r] & \pi_{n+3,n}(\MSL) \ar[r] \ar[rd] & \pi_{n+3,n}(\MML) \ar[r]^-{\partial_*} \ar[d] & \pi_{n+1,n-1}(\MGL) \ar[r] & 0 \\ & & \pi_{n+3,n}(\mathrm{kq}). &  & } \]
    The kernel of the homomorphism $\pi_{n+3,n}(\MML)\to \pi_{n+3,n}(\mathrm{kq})\oplus \pi_{n+1,n-1}(\MGL)$ is given by the intersection of the kernels of the middle vertical map and $\partial_*$. Applying the Snake lemma, we see that this intersection coincides with the kernel of the diagonal homomorphism, which by \cite[Theorem~B(2)]{MSL-slices} is isomorphic to the quotient of $\mathrm{K}^{\mathrm{M}}_{3-n}(F)$ by the subgroup $\partial^2_\infty\mathrm{Sq}^6\mathrm{pr}^\infty_2H^{-n-4,-n}$. It is also easy to see from the diagram that the homomorphism $\pi_{n+3,n}(\MML)\to \pi_{n+3,n}(\mathrm{kq})\oplus \pi_{n+1,n-1}(\MGL)$ is surjective using that $\partial_*$ and $\pi_{n+3,n}(\MSL)\to \pi_{n+3,n}(\mathrm{kq})$ are surjective. Therefore, we obtain a short exact sequence
    \[0\to\bigslant{\mathrm{K}^\mathrm{M}_{3-n}(F)}{\Bigl(\partial^2_\infty\mathrm{Sq}^6\mathrm{pr}^\infty_2H^{-n-4,-n}\Bigr)}\to \pi_{n+3,n}(\MML)\to\pi_{n+3,n}(\mathrm{kq})\oplus \pi_{n+1,n-1}(\MGL)\to 0.\]
    We claim that the $\Einf$-ring map $\MGL\to \mathrm{kgl}$ induces an exact sequence 
    \[ 0\to \mathrm{K}^\mathrm{M}_{3-n}(F)\to \pi_{n+1,n-1}(\MGL)\to\pi_{n+1,n-1}(\mathrm{kgl})\to 0.\]
    First, let us prove that the claim implies the result. Denote by $\mathrm{Ker}_n$ the kernel of the epimorphism $\pi_{n+3,n}(\MML)\twoheadrightarrow\pi_{n+3,n}(\mathrm{kq})\oplus \pi_{n+1,n-1}(\mathrm{kgl})$. The claim and the short exact sequence proved above  together with the Snake lemma imply that there is an extension
    \[ 0\to \bigslant{\mathrm{K}^\mathrm{M}_{3-n}(F)}{\Bigl(\partial^2_\infty\mathrm{Sq}^6\mathrm{pr}^\infty_2H^{-n-4,-n}\Bigr)} \to \mathrm{Ker}_n \to \mathrm{K}^\mathrm{M}_{3-n}(F)\to 0. \]
    Taking the sum over $n$ produces an extension of graded modules over the Milnor--Witt K-theory. This extension splits, since
    \[ \mathrm{Ext}_{\K^{\mathrm{MW}}_{*}(F)}^1(\mathrm{K}^\mathrm{M}_{3+*}(F),M)=0 \]
    for any $\K^{\mathrm{MW}}_{*}(F)$-module $M$ with $M_{-4}=0$ by \cite[Lemma A.3]{Oliver_Suslin}. Any choice of a splitting yields the result.

    It remains to prove the claim. Consider the slice spectral sequences for $\MGL$ and $\mathrm{kgl}$. The morphism $\MGL\to\mathrm{kgl}$ induces an isomorphism on $\slice_q$ for $q\leq 1$ and the map $\begin{pmatrix}
        1 & 0
    \end{pmatrix}\colon \Sigma^{2}_{\Proj^1}\mathrm{M}\Z\{a_1^2\}\oplus\Sigma^{2}_{\Proj^1}\mathrm{M}\Z\{a_2\}\to\Sigma ^{2}_{\Proj^1}\mathrm{M}\Z\{\beta^2\}$ on $\slice_2$ (here $a_i$ is the $i$-polynomial generator of the Lazard ring and $\beta$ is the Bott element), see the proof of Proposition \ref{prop:pi_2_MML}.
    For both $\MGL$ and $\mathrm{kgl}$, the slices $\slice_q$ with $q\geq 3$ do not contribute to $\pi_{n+1,n-1}$. 
    The computation of the first slice differential for $\MGL$ (see \cite[Theorem~4.2]{MSL-slices}) shows that the group $\mathrm{K}^\mathrm{M}_{3-n}(F)$, which comes from the summand $\Sigma^{2}_{\Proj^1}\mathrm{M}\Z\{a_2\}$, survives to the $E^2$-page of the slice spectral sequence. 
    The only higher slice differential that may affect $\pi_{n+1,n-1}(\MGL)$ (resp. $\pi_{n+1,n-1}(\mathrm{kgl})$) is $d_2\colon E^2_{n+2,0,n-1}\to E^2_{n+1,2,n-1}$, which is trivial by \cite[Corollary 3.23]{RSO24} since the unit map $\sph\to \MGL$ (resp. $\sph\to\mathrm{kgl}$) induces an isomorphism on $E^2_{n+2,0,n-1}$ (combine \cite[Theorem 4.2]{MSL-slices} with \cite[Corollary 2.20 and Lemma 3.1]{RSO24}). The claim follows from the strong convergence of the slice spectral sequences for $\MGL$ and $\mathrm{kgl}$.
\end{proof}
\begin{remark}
    The above answers are given in terms of the first Milnor--Witt stems of the very effective algebraic and hermitian K-theory spectra. The first Milnor--Witt stems of $\mathrm{kq}$ are computed (up to extensions) via the slice spectral sequence in \cite[Proposition 2.5, Theorem 2.6]{RSO24} and \cite[Proposition 4.11, Corollary 5.4, Lemmas 5.5, and 5.7, Remarks 6.6 and 6.15]{MSL-slices}. The first Milnor--Witt stems of $\mathrm{kgl}$ are computed (up to extensions) in the proofs of Propositions \ref{prop:pi_1_MML}, \ref{prop:pi_2_MML}, \ref{prop:pi_3_MML}.
\end{remark}
\begin{remark}
    We choose to work with homotopy groups rather than homotopy sheaves for simplicity. In fact, by Morel's theorem \cite[Theorem 2.11]{Morel_hm}, the stated results hold for homotopy sheaves as well. Consequently, we obtain the computation of the homotopy modules $\underline{\pi}_{n}(\MML)_*$ for $n\leq 3$.
\end{remark}
\subsection{Geometric diagonal}\label{section_5.2}
Our approach for a computation of the geometric diagonal (also called the Chow zero line) of $\MML$ closely follows the second author's work on the analogous result for $\MSL$, see \cite{Egor}. In this subsection we assume that the characteristic of $F$ is different from $2$ (the case $\mathrm{char}(F)=2$ is considered separately in Appendix \ref{appendix_char2}). Below, ${}_\eta\pi_{2*,*}(\MML)$ denotes the annihilator of $\eta$ in $\pi_{2*,*}(\MML)$.
\begin{lemma}\label{lemma:quotient_by_eta_tors}
    Let $F$ be a field of characteristic zero. Then the morphism $\MML\to \MML[\eta^{-1}]\cong \MSL[\eta^{-1}]$ induces an isomorphism of the graded $\GW(F)$-algebras (see \cite[Corollary 1.3(3)]{BHop}) $$ \bigslant{\pi_{2*,*}(\MML)}{{}_\eta\pi_{2*,*}(\MML)}\cong \W(F)[u_4,u_8,\dots],\ \text{where}\ \mathrm{deg}(u_{4i})=(8i,4i).$$
    If $F$ is a field of characteristic $p>2$, then the same result holds away from $p$.
\end{lemma}
\begin{proof}
    Below we implicitly invert the exponential characteristic. The morphism $\MML\to\MML[\eta^{-1}]$ induces an epimorphism on $\pi_{2*,*}$ by Theorem \ref{thm:main_cofib_seq} and \cite[Theorem 5.2]{Egor}. Hence, it suffices to prove that $\eta$-primary torsion in $\pi_{2*,*}(\MML)$ is actually $\eta$-torsion. Let $\alpha\in\pi_{2*,*}(\MML)$ be an element such that $\eta^m\cdot\alpha=0$, and consider the following commutative diagram:
    \[ \xymatrix{0 \ar[r] & \pi_{2*,*}(\MSL) \ar[r] \ar[d]^-{\eta} & \pi_{2*,*}(\MML) \ar[r] \ar[d]^-{\eta} & \pi_{2*-2,*-1}(\MGL) \ar[r] \ar[d]^-{\eta} & 0 \\
    0 \ar[r] & \pi_{2*+1,*+1}(\MSL) \ar[r] & \pi_{2*+1,*+1}(\MML) \ar[r] & \pi_{2*-1,*}(\MGL) \ar[r] & 0.
    } \]
    Here the right bottom group vanishes by \cite[Theorem 2.1]{LYZ} and, hence, the left bottom horizontal map is an isomorphism. Thus, if $\eta\cdot\alpha\neq 0$, then its image in the $\eta$-periodization is non-trivial by \cite[Theorem 5.2]{Egor}, which contradicts $\eta^m\cdot\alpha=0$.
\end{proof}

\begin{notation}
    For a field $F$, we denote by $\mathrm{I}_{\MML}(F)$ the graded subgroup of $\pi_{2*,*}(\MML)$ defined as 
    \[\mathrm{I}_{\MML}(F)_n := \begin{cases} \eta \cdot \pi_{2n-1,n-1}(\MML), & \text{if } 4 \mid n, \\ 0, & \text{otherwise}. \end{cases}\] 
    This subgroup captures the fundamental ideals in $\pi_{2*,*}(\MML)$ (at least away from the characteristic).
\end{notation}

\begin{remark}\label{remark_I_MML}
\begin{enumerate}
    \item The morphism $\MSL\to\MML$ induces an isomorphism $\mathrm{I}_{\MSL}(F)\cong \mathrm{I}_{\MML}(F)$; see \cite[\S 6.3]{Egor} for the definition of $\mathrm{I}_{\MSL}(F)$. Indeed, both subgroups are trivial in degrees not divisible by $4$ and in the remaining case, multiplying the exact sequence $$0\to\pi_{2*-1,*-1}(\MSL)\to\pi_{2*-1,*-1}(\MML)\to \pi_{2*-3,*-2}(\MGL)\to 0$$ by $\eta$ and using the fact that $\eta$ acts trivially on $\MGL$, we see that the $\eta$-multiples coincide.
    \item We claim that the graded subgroup $\mathrm{I}_{\MML}(F)$ becomes a graded ideal after inverting the exponential characteristic $\neq 2$. Indeed, choose a quadratically closed field $L/F$ and consider the base change of the short exact sequence $$0\to \pi_{2*,*}(\MSL)\to \pi_{2*,*}(\MML)\to \pi_{2*-2,*-1}(\MGL)\to 0$$
    along this extension (here we omit inversion of the exponential characteristic for simplicity). Then using the computation $\pi_{2*,*}(\MGL)\cong \mathbb{L}$ \cite[Proposition 8.2]{Hoy15}, the previous part of the remark, and the proof of \cite[Lemma 6.13]{Egor}, we obtain that $\mathrm{I}_{\MML}(F)$ is the kernel of the ring homomorphism $\pi_{2*,*}(\MML_F)\to \pi_{2*,*}(\MML_L)$.
\end{enumerate}
\end{remark}

\begin{lemma}\label{lemma:rigidity}
    Let $F$ be a field of characteristic zero. Then the quotient graded ring $$ \bigslant{\pi_{2*,*}(\MML)}{\mathrm{I}_{\MML}(F)}$$ is rigid in the sense that for any extension $L/F$, the base change induces an isomorphism on the corresponding quotient rings. The same rigidity statement holds in characteristic $p>2$ away from $p$.
\end{lemma}
\begin{proof}
    This follows from the main cofiber sequence by means of the first item of Remark \ref{remark_I_MML} and the analogous rigidity statement for $\MSL$ (see \cite[Proposition 6.14]{Egor}).
\end{proof}

\begin{remark}\label{remark:rigidity_dvr}
    More generally, the above computations are valid for a local Dedekind domain $R$ with exponential residue characteristic $e\neq 2$, away from $e$ (including the mixed characteristic case). This holds because the results of \cite{Egor} are valid in this generality (for the case of $\MGL$ see \cite{SpiMGL}) and our main cofiber sequence holds over an arbitrary qcqs scheme. In particular, Lemma \ref{lemma:rigidity} shows that the corresponding quotient rings are rigid with respect to a homomorphism of local Dedekind domains whose exponential residue characteristics are different from $2$ (away from the exponential residue characteristic of the target).
\end{remark}

Now, in order to relate the rigid part of the answer to its topological counterpart, we show that the complex realization of $\MML$ is equivalent to  Stong's complex-spin cobordism spectrum $\mathrm{M}\Sigma\in \Sp$ \cite{Stong}, see also Appendix \ref{appendix_Stong} for an overview. Recall that the complex Betti realization is the symmetric monoidal functor $\realiz_{\mathrm{B}\C}\colon \SH(\C)\to \Sp$. 
\begin{proposition}\label{prop:complex_realization_MML}
    There is an equivalence of $\Einf$-ring spectra $\realiz_{\mathrm{B}\C}(\MML_\C)\cong\mathrm{M}\Sigma$.
\end{proposition}
\begin{proof}
    Using the fact that the maximal compact subgroup of $\ML_n(\C)$ is $\tilde{\mathrm{U}}(n)$ (see Notation \ref{notation:two_fold_covering_U(n)}) and repeating the argument of \cite[Theorem C.7]{MG_realization}, we obtain an equivalence of $\Einf$-ring spectra $\realiz_{\mathrm{B}\C}(\MML_\C)\cong \mathrm{M}\tilde{\mathrm{U}}$. To conclude, it remains to apply Corollary \ref{corollary:model_stong_spectrum}.
\end{proof}
\begin{lemma}\label{lemma:quotient_by_fund_ideals}
    Let $F$ be a field of characteristic zero. Then there is an isomorphism of graded $\GW(F)$-algebras
    $$ \bigslant{\pi_{2*,*}(\MML)}{\mathrm{I}_{\MML}(F)}\cong \pi_{2*}(\mathrm{M}\Sigma), $$
    where the action of $\GW(F)$ on the right hand side is induced by the rank homomorphism.
    For $F=\C$ the complex Betti realization induces such an isomorphism. If $F$ is a field of characteristic $p>2$, then the same result holds away from $p$.
\end{lemma}
\begin{proof}
    Consider first the case of $F=\C$. The complex Betti realization induces the following morphism of short exact sequences:
    \[ \xymatrix{0 \ar[r] & \pi_{2*,*}(\MSL_\C) \ar[r] \ar[d]^-\simeq & \pi_{2*,*}(\MML_\C) \ar[r] \ar[d] & \pi_{2*-2,*-1}(\MGL_\C) \ar[r] \ar[d]^-\simeq & 0 \\ 
    0 \ar[r] & \pi_{2*}(\MSU) \ar[r] & \pi_{2*}(\mathrm{M}\Sigma) \ar[r] & \pi_{2*-2}(\MU) \ar[r] & 0,
    }\]
    where the right vertical homomorphism is an isomorphism by \cite[Proposition 8.2]{Hoy15} and the left vertical map is bijective by \cite[Theorem 6.16]{Egor}. It follows that the desired isomorphism holds over $\C$. Using the rigidity property established in Lemma \ref{lemma:rigidity}, we obtain the result over any field of characteristic zero. 

    To prove the result in positive characteristic $p>2$, consider a discrete valuation ring $R$ of mixed characteristic $(0,p)$ whose residue field is $F$ (if $F$ is perfect, one may take the ring of $p$-adic Witt vectors $\mathbb{W}_{p^\infty}(F)$; otherwise, the Cohen ring of $F$ works \cite[\S 6.19, 6.20, 6.23]{HAZEWINKEL}). By Remark \ref{remark:rigidity_dvr}, the base changes along the homomorphisms of local Dedekind domains $F\twoheadleftarrow R \hookrightarrow \mathrm{Frac}(R)=:L$ induce an isomorphism
    $$ \bigslant{\pi_{2*,*}(\MML_F)}{\mathrm{I}_{\MML}(F)}[\nicefrac{1}{p}]\cong \bigslant{\pi_{2*,*}(\MML_{L})}{\mathrm{I}_{\MML}(L)}[\nicefrac{1}{p}]. $$
    Hence, we reduce to the case of a field of characteristic zero, which has already been established.
\end{proof}

\begin{theorem}\label{thm:pullback_descr}
    Let $F$ be a field of characteristic zero. Then the following diagram is a pullback of graded $\GW(F)$-algebras:
        \[\begin{tikzcd} {\pi_{2*,*}(\MML)} \arrow[dr, phantom, "\lrcorner", very near start] \arrow[dd, two heads] \arrow[rr, two heads] & & {\pi_{2*}(\mathrm{M}\Sigma)} \arrow[dd, two heads] \\  & {} & & \\{\W(F)[u_4,u_8,\dots]} \arrow[rr, "\overline{\mathrm{rk}}", two heads] &  & {\Z/2[u_4,u_8,\dots]},\end{tikzcd}\]
    where the top horizontal homomorphism is the map from Lemma \ref{lemma:quotient_by_fund_ideals}, the left vertical homomorphism is the quotient by the annihilator of $\eta$, and the right one is the quotient by the annihilator of $\etatop\in\pi_{1}(\mathrm{M}\Sigma)$. If $F$ is a field of characteristic $p>2$, then the same result holds away from $p$.
\end{theorem}
\begin{proof}
    We omit inversion of the exponential characteristic throughout the proof. The quotient of $\pi_{2*,*}(\MML)$ by the intersection $\mathrm{I}_{\MML}(F)\cap {}_\eta\pi_{2*,*}(\MML)$ is isomorphic to the pullback of the diagram
    $$ \bigslant{\pi_{2*,*}(\MML)}{\mathrm{I}_{\MML}(F)}\to \bigslant{\pi_{2*,*}(\MML)}{(\mathrm{I}_{\MML}(F)+{}_\eta\pi_{2*,*}(\MML))}\leftarrow \bigslant{\pi_{2*,*}(\MML)}{{}_\eta\pi_{2*,*}(\MML)}. $$
    Applying Lemma \ref{lemma:quotient_by_eta_tors} and Lemma \ref{lemma:quotient_by_fund_ideals}, we obtain that this diagram is isomorphic to
    $$ \pi_{2*}(\mathrm{M}\Sigma)\to \Z/2[u_4,u_8,\dots]\leftarrow\W(F)[u_4,u_8,\dots]. $$
    A straightforward verification shows that the maps in the diagram are as required (cf. discussion just before \cite[Theorem 6.17]{Egor}). It remains to note that the intersection $\mathrm{I}_{\MML}(F)\cap {}_\eta\pi_{2*,*}(\MML)$ is trivial: $\mathrm{I}_{\MML}(F)$ maps injectively into $\eta$-periodization, while ${}_\eta\pi_{2*,*}(\MML)$ is annihilated by $\eta$.
\end{proof}

\begin{corollary}
    Let $F$ be a quadratically closed field of characteristic zero. There is an isomorphism
    $$ \pi_{2*,*}(\MML)\cong \pi_{2*}(\mathrm{M}\Sigma). $$
    If $F$ is a quadratically closed field of characteristic $p>2$, then the same result holds away from $p$.
\end{corollary}
\begin{proof}
    This follows directly from Theorem \ref{thm:pullback_descr}, since for a quadratically closed field the bottom epimorphism is an isomorphism.
\end{proof}

Since Voevodsky's algebraic cobordism spectrum $\MGL$ is $\eta$-complete, the $\Einf$-ring map $\MML\to\MGL$ factors through the $\eta$-completion of $\MML$.

\begin{corollary}\label{corollary:geom_diag_away_2}
    Assume that $F$ is a field of characteristic zero. Then the morphisms $\MML\to \MML^\wedge_\eta\to \MGL$ and $\MML\to \MML[\eta^{-1}]\cong\MSL[\eta^{-1}]$ induce isomorphisms of graded $\GW(F)$-algebras $$\pi_{2*,*}(\MML)[\nicefrac{1}{2}]\cong 
    \pi_{2*,*}(\MGL)[\nicefrac{1}{2}]\times \pi_{2*,*}(\MSL[\eta^{-1}])[\nicefrac{1}{2}]\cong \mathbb{L}[\nicefrac{1}{2}]\times \W(F)[\nicefrac{1}{2}][u_4,u_8,\dots],$$
    where $\mathbb{L}$ is the Lazard ring and $\mathrm{deg}(u_{4i})=(8i,4i)$. If $F$ is a field of characteristic $p>2$, then the same result holds away from $p$.
\end{corollary}
\begin{proof}
    Below we implicitly invert the exponential characteristic. The second isomorphism is proven in \cite[Proposition 8.2]{Hoy15} and \cite[Corollary 1.3(3)]{BHop}, and we need to prove the first isomorphism. We have the following standard decomposition (see e.g., \cite[Remark 4]{ALP}):
    \begin{align*}
        \pi_{2*,*}(\MML)[\nicefrac{1}{2}]&\cong\pi_{2*,*}(\MML)[\nicefrac{1}{2}]^+\times \pi_{2*,*}(\MML)[\nicefrac{1}{2}]^- \\ 
        & \cong\pi_{2*,*}(\MML^\wedge_\eta)[\nicefrac{1}{2}]\times \pi_{2*,*}(\MML[\eta^{-1}])[\nicefrac{1}{2}].
    \end{align*}
    Here the plus (resp. minus) part of the homotopy groups is by definition the corresponding homotopy groups of the plus (resp. minus) part. The statement about the minus part follows from Corollary \ref{cor_eta_period_mslc}. Let us consider the plus part. The pullback description implies that
    $$ \pi_{2*,*}(\MML^\wedge_\eta)[\nicefrac{1}{2}]\cong \pi_{2*}(\mathrm{M}\Sigma)[\nicefrac{1}{2}]. $$
    Hence, by the proof of Lemma \ref{lemma:quotient_by_fund_ideals}, it suffices to show that 
    \begin{enumerate}
        \item  the rings $\pi_{2*,*}(\MML^\wedge_\eta)[\nicefrac{1}{2}]$ and $\pi_{2*,*}(\MGL)[\nicefrac{1}{2}]$ are stable under base change along a homomorphism of local Dedekind domains (with exponential residue characteristic of the target $e\neq 2$, away from $e$),
        \item the morphism $\mathrm{M}\Sigma\to \MU$ induces an isomorphism $\pi_{2*}(\mathrm{M}\Sigma)[\nicefrac{1}{2}]\xrightarrow{\simeq}\pi_{2*}(\MU)[\nicefrac{1}{2}]$.
    \end{enumerate}
    The first statement follows from Lemma \ref{lemma:rigidity} and Remark \ref{remark:rigidity_dvr} (the ideal $\mathrm{I}_\MML(F)$ belongs to the minus part), together with the computation $\pi_{2*,*}(\MGL)\cong \mathbb{L}$ (see \cite[Theorem 6.5]{SpiMGL}). The second one is exactly item $3$ of Theorem \ref{theorem:homotopy_groups_MSigma}.
\end{proof}
\begin{remark}
    The stated pullback description of $\pi_{2*,*}(\MML)$ (as well as Corollary \ref{corollary:geom_diag_away_2}) holds more generally over a local Dedekind domain $R$ with exponential residue characteristic $e$, away from $e$ (the case $e=2$ can be included using Appendix \ref{appendix_char2}, equivalence \eqref{equa_bach_equiv_minus}, and the computation of $\pi_{*}(\mathrm{MSO})[\nicefrac{1}{2}]$). We leave the details to the interested reader.
\end{remark}

\section{Slices and the category of \texorpdfstring{$2$}{2}-inverted modules}\label{section:6}
In this section, we study the stable category of $2$-inverted $\MML$-modules over a field (away from the characteristic). For this purpose, we compute the slices of $\MML$ and use them to compare $\MML^\wedge_\eta[\nicefrac{1}{2}]$ with $\MGL[\nicefrac{1}{2}]$. We then categorify this comparison, which yields a description of the plus part. This is followed by a characterization of the minus part via the real étale realization.
\subsection{Comparison with \texorpdfstring{$\MGL$}{MGL}}\label{subsection:6-1}
Below we use Voevodsky's slice filtration; see \cite[\S 2]{RO16} for an overview. We denote by $\slice_q(-)$ the $q$-th slice functor. Given a spectrum $\EE\in\SHS$, the slice filtration of $\EE$ induces a conditionally convergent spectral sequence
\[ E^1_{p,q,w}(\EE)=\pi_{p,w}(\slice_q(\EE)) \Longrightarrow \pi_{p,w}(\mathrm{sc}(\EE)), \]
where $\mathrm{sc}(\EE)$ is a motivic spectrum, called the \textit{slice completion of} $\EE$, see \cite[Definition 3.1]{RSO19}.
\begin{lemma}\label{lemma_slice_convergence}
    We have $\mathrm{sc}(\MML)\cong \MML^\wedge_\eta$. In other words, there is a conditionally convergent slice spectral sequence
    \[ E^1_{p,q,w}(\MML)=\pi_{p,w}(\slice_q(\MML)) \Longrightarrow \pi_{p,w}(\MML^\wedge_{\eta}). \]
\end{lemma}
\begin{proof}
    Since $\MML$ is effective, it suffices to show that $\MML/\eta$ is $\eta$-complete by \cite[Lemma~3.13]{RSO19}. By Theorem \ref{thm:main_cofib_seq}, we need to prove that $\MSL/\eta$ and $\Sigma^{1}_{\Proj^1}\MGL/\eta$ are $\eta$-complete. The $\eta$-completeness of the second spectrum is well known (note that $\MGL/\eta\cong \MGL\oplus \Sigma^{1}_{\Proj^1}\MGL$), while for the first spectrum this is proven in \cite[Lemma~6.5]{MSL-slices}.
\end{proof}
We now compute the slices of $\MML$ over a field away from the characteristic.
\begin{proposition}\label{prop:slices_MML}
    Suppose that $F$ is a field of characteristic zero. The slices of $\MML$ are given by 
    \[\slice_q(\MML)\cong \bigoplus_{0\leq s<\nicefrac{q}{2}} \Sigma^{q+2s,q}\mathrm{M}(\Z/2)^{p(\lfloor\nicefrac{s}{2}\rfloor)}\oplus \Sigma^{2q,q}\mathrm{M}\Z^{p(q)}, \] 
    where $p(-)$ is the partition function and $\lfloor-\rfloor$ is the floor function. In positive characteristic, the same result holds after inverting the characteristic.
\end{proposition}
\begin{proof}
    We implicitly invert the exponential characteristic in the proof. Since slices commute with finite direct sums, we have 
    \[ \slice_q(\MML)\cong \slice_q(\MSL)\oplus \slice_q(\Sigma^{1}_{\Proj^1}\MGL)\cong \slice_q(\MSL)\oplus \Sigma^{1}_{\Proj^1}\slice_{q-1}(\MGL). \]
    Here the first isomorphism is induced by the choice of $t$ in Theorem \ref{thm:main_cofib_seq}, and the second equivalence follows from \cite[Lemma 2.1]{RO16}. Combining this decomposition with the computations of $\slice_q(\MSL)$ \cite[Theorem~A]{MSL-slices} and $\slice_{q-1}(\MGL)$ \cite[Theorem 3.1]{SpiMGL}, we obtain the result.
\end{proof}
\begin{remark}\label{remark:slices_AN}
    The conceptual explanation of the formula above is (here we omit inversion of the exponential characteristic of $F$):
    \[\slice_q(\MML)\cong \bigoplus_{s=0}^q \Sigma^{2q-s,q} \mathrm{M}\bigl(\mathrm{Ext}_{\MU^*\MU}^{s,2q}(\MU^*(\mathrm{M}\Sigma),\MU^*)\bigr). \]
    This follows from the splitting $\MML\cong \MSL\oplus \Sigma^1_{\Proj^1}\MGL$, analogous splitting in topology $\mathrm{M}\Sigma\cong \MSU\oplus \Sigma^{2}\MU$ and a similar view on the slices of $\MSL$, see \cite[Remark 3.11]{MSL-slices}.
\end{remark}
\begin{proposition}\label{prop_mslc_mgl}
    Let $F$ be a field of characteristic zero. Then the canonical map of motivic $\Einf$-ring spectra over $F$ \[ \MML^\wedge_\eta[\nicefrac{1}{2}]\to \MGL[\nicefrac{1}{2}]\] is an equivalence. In characteristic $p>0$, the same result holds away from $p$.
\end{proposition}
\begin{proof}
    We implicitly invert the exponential characteristic of $F$ below. Since all ingredients are stable under base change, passing to the perfection, we can assume that $F$ is perfect; see \cite[Corollary~2.1.7]{EK-perfection}.
    
    We claim that the morphism $\MML[\nicefrac{1}{2}]\to \MGL[\nicefrac{1}{2}]$ induces an isomorphism on all slices. From the shape of the slices, we see that the morphism $\slice_q(\MML)[\nicefrac{1}{2}]\to \slice_q(\MGL)[\nicefrac{1}{2}]$ is an equivalence if and only if it induces an isomorphism after applying $\pi_{2q,q}(-)$. There are the following slice spectral sequences (see \cite[8.14]{Hoy15} and Lemma \ref{lemma_slice_convergence}):
    \begin{align*} & E^1_{p,q,w}(\MML)[\nicefrac{1}{2}]=\pi_{p,w}(\slice_q(\MML))[\nicefrac{1}{2}] \Longrightarrow \pi_{p,w}(\MML^\wedge_\eta)[\nicefrac{1}{2}], \\
    & E^1_{p,q,w}(\MGL)[\nicefrac{1}{2}]=\pi_{p,w}(\slice_q(\MGL))[\nicefrac{1}{2}] \Longrightarrow \pi_{p,w}(\MGL)[\nicefrac{1}{2}].
    \end{align*}
    A direct argument using the standard vanishing of motivic cohomology (cf. \cite[Theorem 2.1]{LYZ}) shows that the groups $E^1_{2q,q,q}(\MML)[\nicefrac{1}{2}]$ and $E^1_{2q,q,q}(\MGL)[\nicefrac{1}{2}]$ survive to the $E^\infty$-pages. Moreover, they are the only groups that contribute to $\pi_{2q,q}$. Thus we have the following commutative diagram:
    \begin{center}
        \begin{tikzcd}[column sep=small]
            {\pi_{2q,q}(\slice_q(\MML))[\nicefrac{1}{2}]} \arrow[d] \arrow[r, Rightarrow, no head] & {E^1_{2q,q,q}(\MML)[\nicefrac{1}{2}]} \arrow[r, Rightarrow, no head] \arrow[d] & {E^\infty_{2q,q,q}(\MML)[\nicefrac{1}{2}]} \arrow[d] & {\pi_{2q,q}(\MML^\wedge_\eta)[\nicefrac{1}{2}]} \arrow[d, "\simeq"] \arrow[l, "\simeq"'] \\
            {\pi_{2q,q}(\slice_q(\MGL))[\nicefrac{1}{2}]} \arrow[r, Rightarrow, no head]            & {E^1_{2q,q,q}(\MGL)[\nicefrac{1}{2}]} \arrow[r, Rightarrow, no head]            & {E^\infty_{2q,q,q}(\MGL)[\nicefrac{1}{2}]}   & {\pi_{2q,q}(\MGL)[\nicefrac{1}{2}].} \arrow[l, "\simeq"']
        \end{tikzcd}
    \end{center}
    The right vertical map is an isomorphism by Corollary \ref{corollary:geom_diag_away_2} and Corollary \ref{corollary:MML_MGL_char2}. Therefore, all vertical maps are isomorphisms and the claim follows. 
    
    The isomorphism on the slices induces an isomorphism of the $E^1$-pages:
    \[ E^1_{p,q,w}(\MML)[\nicefrac{1}{2}]\xrightarrow{\simeq} E^1_{p,q,w}(\MGL)[\nicefrac{1}{2}] \enspace \text{for all}\ p,q,w\in\Z. \]
    Consequently, the map $\MML[\nicefrac{1}{2}]\to \MGL[\nicefrac{1}{2}]$ induces an isomorphism between the $E^\infty$-pages of the above spectral sequences. From the strong convergence of these spectral sequences, which holds because for each tuple $(p,q,w)$ there are only finitely many non-zero entering and exiting differentials, we deduce that the desired morphism induces an isomorphism on all homotopy groups. Finally, by Morel's theorem about strictly $\A^1$-invariant Nisnevich sheaves (see \cite[Theorem 2.11]{Morel_hm}), this isomorphism also holds at the level of homotopy sheaves (alternatively, one can conclude using the fact that both spectra are cellular).
\end{proof}
\begin{remark}
    Suppose that $F$ is a field of characteristic zero and $\EE$ is an $\eta$-complete $\ML$-oriented homotopy commutative ring spectrum over $F$. Then, by Proposition \ref{prop_mslc_mgl}, $\EE[\nicefrac{1}{2}]$ admits a $\GL$-orientation (i.e., $\EE[\nicefrac{1}{2}]$ is oriented in the usual sense). Indeed, choose an orientation morphism $\MML\to\EE$ (see Remark \ref{remark:univ_prop_MML}) and consider the following diagram:
    \begin{center}
        \begin{tikzcd}
            \Th_{\Proj^\infty}(\struct(-2)) \arrow[d, "\thom(\struct(-2))"'] \arrow[rr, "\Th(\nu_2)"]       &                        & \Th_{\Proj^\infty}(\struct(-1)) \arrow[d, "\thom(\struct(-1))"] \\
            {\Sigma^{1}_{\Proj^1}\MML^\wedge_\eta[\nicefrac{1}{2}]} \arrow[rd] \arrow[rr, "\simeq"] &                        & {\Sigma^1_{\Proj^1}\MGL[\nicefrac{1}{2}]} \arrow[ld, dashed]        \\
            & {\Sigma^1_{\Proj^1}\EE[\nicefrac{1}{2}]}, &                                      
        \end{tikzcd}
    \end{center}
    where $\Th(\nu_2)$ is the map on the Thom spectra induced by the degree $2$ Veronese map $\nu_2\colon\Proj^\infty\to\Proj^\infty$ and $\thom(\struct(-2))$ (resp. $\thom(\struct(-1))$) is the Thom class of $\struct(-2)$ (resp. $\struct(-1)$). Then composing the dashed morphism with $\thom(\struct(-1))$, we obtain an element that defines the $\GL$-orientation of $\EE[\nicefrac{1}{2}]$. Note that a priori we cannot claim that the dashed morphism gives us a homotopy ring morphism after $\Proj^1$-desuspension, since $\MML\to\EE$ may not be multiplicative. Nevertheless, the above information is enough. In positive characteristic $p>0$, the same conclusion holds if $\EE\in\SH(F)[\nicefrac{1}{p}]$.
\end{remark}
\subsection{Category of modules}
We are now ready to consider the category of $\MML$-modules. First, we note that, as usual, the category of $2$-inverted $\MML$-modules splits as a symmetric monoidal $\infty$-category into the product
\begin{equation}
    \mathrm{Mod}_{\MML}(\SH(F))[\nicefrac{1}{2}]\cong \mathrm{Mod}_{\MML}(\SH(F))[\nicefrac{1}{2}]^+\times \mathrm{Mod}_{\MML}(\SH(F))[\nicefrac{1}{2}]^-.
\end{equation}
The projection of the motivic Hopf element in the plus part is trivial $(\MML\wedge\eta)[\nicefrac{1}{2}]^+=0$, while its minus part $(\MML\wedge\eta)[\nicefrac{1}{2}]^-$ is invertible. Consequently, $\MML[\nicefrac{1}{2}]^+$ coincides with the $\eta$-completion $\MML^\wedge_\eta[\nicefrac{1}{2}]$, which is already known to be equivalent to $\MGL[\nicefrac{1}{2}]$.
\begin{theorem}\label{thm_plus_part}
    Let $F$ be a field of characteristic zero. Then the morphism $\MML\to\MGL$ induces an equivalence of symmetric monoidal $\infty$-categories
    \[\mathrm{Mod}_{\MML}(\SH(F))[\nicefrac{1}{2}]^+\cong\mathrm{Mod}_{\MGL}(\SH(F))[\nicefrac{1}{2}]. \]
    For a field $F$ of characteristic $p>0$ the same result holds after inverting $p$.
\end{theorem}
\begin{proof}
    We implicitly invert the characteristic below. Consider the following adjunction
    \[ \SH(F)\rightleftarrows \mathrm{Mod}_{\MML}(\SH(F))[\nicefrac{1}{2}]^+, \]
    given by the composition of the free-forgetful adjunction with a left Bousfield localization.
    By abstract nonsense (see \cite[Construction 5.23]{MNN-mod-cat}), this adjunction induces the adjunction
    \[ \mathrm{Mod}_{\MML[\nicefrac{1}{2}]^+}(\SH(F))\rightleftarrows \mathrm{Mod}_{\MML}(\SH(F))[\nicefrac{1}{2}]^+, \]
    which is an equivalence of symmetric monoidal $\infty$-categories by \cite[Proposition 5.29]{MNN-mod-cat}. Now the left-hand side is equivalent to the category of modules over $\MGL[\nicefrac{1}{2}]$ by Proposition \ref{prop_mslc_mgl}. It remains to identify $\MGL[\nicefrac{1}{2}]$-modules with $2$-inverted $\MGL$-modules. This can be done using the same argument applied to the adjunction $\SH(F)\rightleftarrows\mathrm{Mod}_{\MGL}(\SH(F))[\nicefrac{1}{2}]$.
\end{proof}
We now describe the minus part. For that we use Bachmann's equivalence \cite{Bachmann_ret} (below, $\rho$ denotes the map $\sph\to\Gm$ corresponding to $-1$):
\[\SH(F)[\rho^{-1}]\cong\Sp(\Sper(F)),\] 
which is induced by the real \'etale realization (see Appendix \ref{appendix_real_etale} for an overview). It follows from \cite[Lemma 39]{Bachmann_ret} that away from $2$, inverting $\rho$ is equivalent to inverting $\eta$. In particular, away from $2$, the real \'etale realization factors through the symmetric monoidal equivalence
\begin{equation}\label{equa_bach_equiv_minus}
    \SH(F)[\nicefrac{1}{2}]^-\cong \Sp(\Sper(F))[\nicefrac{1}{2}].
\end{equation}
\begin{theorem}
    Let $F$ be a field. The real \'etale realization induces an equivalence of the symmetric monoidal $\infty$-categories
    \[ \mathrm{Mod}_{\MML}(\SH(F))[\nicefrac{1}{2}]^-\cong \mathrm{Mod}_{\underline{\mathrm{MSO}}}(\Sp(\Sper(F)))[\nicefrac{1}{2}]. \]
    Here $\underline{\mathrm{MSO}}\in\mathrm{CAlg}(\Sp(\Sper(F)))$ denotes the constant sheaf on the Thom spectrum $\mathrm{MSO}\in\mathrm{CAlg}(\Sp)$.
\end{theorem}
\begin{proof}
    Repeating the same argument as in the proof of Theorem \ref{thm_plus_part}, we obtain the following symmetric monoidal equivalence
    \[ \mathrm{Mod}_{\MML}(\SH(F))[\nicefrac{1}{2}]^-\cong \mathrm{Mod}_{\MML[\nicefrac{1}{2}]^-}(\SH(F)). \]
    In the right-hand category of modules, we may replace $\MML$ by $\MSL$ since they become equivalent as motivic $\Einf$-ring spectra after inverting $\eta$ (see Corollary \ref{cor_eta_period_mslc}). The equivalence \eqref{equa_bach_equiv_minus} then induces the symmetric monoidal equivalence\footnote{Note that the $\infty$-category $\mathrm{Mod}_{\MSL[\nicefrac{1}{2}]^-}(\SH(F))$ is equivalent to $\mathrm{Mod}_{\MSL[\nicefrac{1}{2}]^-}(\SH(F)[\nicefrac{1}{2}]^-)$.}
    \[ \mathrm{Mod}_{\MSL[\nicefrac{1}{2}]^-}(\SH(F))\cong \mathrm{Mod}_{\realiz_{\ret}(\MSL)[\nicefrac{1}{2}]}(\Sp(\Sper(F))).\]
    Thus we need to compute $\realiz_{\ret}(\MSL)$ as a sheaf of $\Einf$-ring spectra. By Example \ref{example_mso}, it is given by $\underline{\mathrm{MSO}}$ equipped with its natural $\Einf$-ring structure. Identifying modules over $\underline{\mathrm{MSO}}[\nicefrac{1}{2}]$ with $2$-inverted $\underline{\mathrm{MSO}}$-modules yields the desired result.
\end{proof}

\begin{remark}
    It follows that if $F$ is not formally real, then the canonical morphism $\MML\to \MGL$ induces an equivalence (here we omit inversion of the exponential characteristic for simplicity): $$\mathrm{Mod}_{\MML}(\SH(F))[\nicefrac{1}{2}]\cong \mathrm{Mod}_{\MGL}(\SH(F))[\nicefrac{1}{2}].$$
    It would be interesting to obtain this equivalence from the description of these categories in terms of transfers given in \cite{Mb-modules}.
    More precisely, the category of $\MGL$-modules admits a description as the category of motivic spectra with transfers along finite syntomic morphisms (see \cite[Theorem~4.1.3]{Mb-modules}), whereas the category of $\MML$-modules corresponds to motivic spectra with transfers along finite syntomic morphisms $f\colon Z\to X$ with a choice $\omega_f \cong \mathcal{L}^{\otimes 2}$ for an invertible sheaf $\mathcal{L}$ over $Z$ (see \cite[Remark 7.11]{HJNY-HermKtheory} and \cite[Remark 4.2.4]{Mb-modules}). Hence, away from $2$, the additional condition $\omega_f \cong \mathcal{L}^{\otimes 2}$ can be omitted when $F$ is not formally real.
\end{remark}
\begin{remark}\label{remark:rat_decompos}
    One can deduce that the canonical map $(\MML_S\otimes\Q)^+\to \MGL_S\otimes\Q$ is an equivalence for an arbitrary qcqs scheme $S$. Indeed, by \cite[Proposition B.3]{BH-Norms} it sufficient to consider the case $S=\Spec(F)$ for a field $F$ (in order to remove the assumption that $S$ is locally of finite Krull dimension, we first reduce by base change to $\Spec(\Z)$ and then to its residue fields), which follows from Theorem~\ref{thm_plus_part}. One can further obtain a decomposition of  the rational metalinear algebraic cobordism spectrum $\MML_S\otimes\Q$ into the product of polynomial homotopy commutative ring spectra:
    \[ \mathrm{M}\Q[a_1,a_2,\dots]\times \mathrm{W}\Q[u_4,u_8,\dots]\xrightarrow{\simeq}\MML_S\otimes \Q, \]
    where $\mathrm{M}\Q$ is the rational motivic cohomology spectrum and $\mathrm{W}\Q$ is the rational Witt motivic cohomology spectrum. Here the generator $a_i$ has bidegree $(2i,i)$ and the generator $u_{4i}$ has bidegree $(8i,4i)$. The decomposition of the plus part follows from the above observation together with the decomposition of $\MGL_S\otimes\Q$ (see \cite[Theorem 10.5]{NSOLandw}; assumptions on the scheme $S$ can be removed by base change from $\Spec(\Z)$), while the decomposition of the minus part follows from Corollary \ref{cor_eta_period_mslc} together with the corresponding decomposition of $(\MSL_S\otimes\Q)^-$ obtained in \cite[Lemma 7.7]{MSL-slices}.
\end{remark}
\appendix
\section{Stong's complex-spin cobordism spectrum}\label{appendix_Stong}
In the paper \cite{Stong}, Stong introduced the complex-spin Thom spectrum in topology and computed its homotopy groups. In this appendix, we give an overview of his results and prove that his spectrum is equivalent to the Thom spectrum built from the $2$-fold coverings of the unitary groups. This comparison is inspired by the classical work of Atiyah \cite{Atiyah}.

\begin{notation}
    Consider the space $\mathrm{B}\Sigma(n)$ defined as the pullback of the following diagram: \[\BU(n)\rightarrow \mathrm{BSO}\leftarrow \mathrm{BSpin},\] where the left map is induced by the inclusion $\mathrm{U}(n)\hookrightarrow \mathrm{SO}(2n)$, and the right map is induced by the $2$-fold covering $\mathrm{Spin}\twoheadrightarrow \mathrm{SO}$. 
    We denote by $\mathrm{B}\Sigma$ the colimit \[\colim_n \mathrm{B}\Sigma(n),\] which is equivalent to the pullback of the diagram $\BU\rightarrow \mathrm{BSO}\leftarrow \mathrm{BSpin}$. Since all maps involved are morphisms of $\Einf$-groups, $\mathrm{B}\Sigma$ inherits a canonical $\Einf$-group structure.
\end{notation}
\begin{definition}[{\cite{Stong}}]
    The \textit{complex-spin Thom spectrum} $\mathrm{M}\Sigma\in \mathrm{CAlg}(\Sp)$ is the Thom spectrum associated with the composite $\mathrm{B}\Sigma\to\mathrm{BU}\to\mathrm{BO}$. By construction, there are forgetful $\Einf$-ring morphisms $\MSU\to\mathrm{M}\Sigma\to \MU$.
\end{definition}
\begin{remark}
    The Thom spectrum $\mathrm{M}\Sigma$ corresponds, under the Pontryagin--Thom construction, to the cobordism theory of manifolds equipped with compatible stable unitary and spinor structures.
\end{remark}
\begin{notation}\label{notation:two_fold_covering_U(n)}
    We denote by $\tilde{\mathrm{U}}(n)$ the kernel of the homomorphism of Lie groups
        \[\mathrm{U}(n)\times \mathrm{U}(1)\to \mathrm{U}(1),\, (u,t)\mapsto t^{-2}\cdot\det(u).\]
    This group is a $2$-fold covering of the unitary group, i.e., there is an exact sequence of Lie groups \[1\to C_2\to \tilde{\mathrm{U}}(n)\to\mathrm{U}(n)\to 1.\] The special unitary group $\mathrm{SU}(n)$ embeds into $\tilde{\mathrm{U}}(n)$ via $u\mapsto (u,1)$. We denote by $\MUtf$ the Thom spectrum associated with the colimit $\mathrm{B}\tilde{\mathrm{U}}:=\colim_n \mathrm{B}\tilde{\mathrm{U}}(n)$ along the composite $\mathrm{B}\tilde{\mathrm{U}}\to\mathrm{BU}\to\mathrm{BO}$. This is an $\Einf$-ring spectrum.
\end{notation}
The inclusion $\mathrm{U}(n)\hookrightarrow\mathrm{SO}(2n)$ lifts to an inclusion of the $2$-fold covering groups $\tilde{\mathrm{U}}(n)\hookrightarrow \mathrm{Spin}(2n)$ (see e.g., \cite[2.11]{BU_tilde_physics}). Hence, we have the following commutative (up to homotopy) diagram 
\[\begin{tikzcd}
	{\mathrm{B}\tilde{\mathrm{U}}(n)} & {\mathrm{BSpin}} \\
	{\mathrm{BU}(n)} & {\mathrm{BSO}.}
	\arrow[from=1-1, to=1-2]
	\arrow[from=1-1, to=2-1]
	\arrow[from=1-2, to=2-2]
	\arrow[from=2-1, to=2-2]
\end{tikzcd}\]
This induces a morphism $\mathrm{B}\tilde{\mathrm{U}}(n)\to \mathrm{B}\Sigma(n)$ from $\mathrm{B}\tilde{\mathrm{U}}(n)$ to the pullback of this diagram.
\begin{proposition}\label{prop:BU_BSigma}
    For $n\geq 1$, the morphism $\mathrm{B}\tilde{\mathrm{U}}(n)\to \mathrm{B}\Sigma(n)$ constructed above is an equivalence.
\end{proposition}
\begin{proof}
    We claim that $\tilde{\mathrm{U}}(n)$ is the pullback of the diagram $\mathrm{U}(n)\to \mathrm{SO}\leftarrow \mathrm{Spin}$ in the $\infty$-category of $\Eone$-groups.
    Indeed, the induced map $\mathrm{fib}(\tilde{\mathrm{U}}(n)\to\mathrm{U}(n))\to\mathrm{fib}(\mathrm{Spin}\to \mathrm{SO})$ is an equivalence, since both fibers are equivalent to $C_2$ and the homomorphism $\pi_1(\mathrm{U}(n))\to \pi_1(\mathrm{SO})$ is surjective. 
    The result follows from the claim and the definition of $\mathrm{B}\Sigma(n)$, since the functor 
    \[\mathrm{B}-\colon \mathrm{Grp}_{\Eone}(\Spc)\to \Spc_{*,\geq 1},\] 
    defined via the usual bar construction, is an equivalence of $\infty$-categories and therefore preserves pullbacks (its inverse is given by the loop functor $\Omega$). 
    Here, $\mathrm{Grp}_{\Eone}(\Spc)$ denotes the $\infty$-category of $\Eone$-groups in spaces and $\Spc_{*,\geq 1}$ denotes the $\infty$-category of pointed connected spaces.
\end{proof}
\begin{corollary}\label{corollary:model_stong_spectrum}
    The map $\mathrm{B}\tilde{\mathrm{U}}\to\mathrm{B}\Sigma$ induces an equivalence of the $\mathbb{E}_\infty$-ring spectra $\MUtf\xrightarrow{\simeq}\mathrm{M}\Sigma$.
\end{corollary}
\begin{proof}
    This follows directly from the Proposition \ref{prop:BU_BSigma}, since the colimit equivalence $\mathrm{B}\tilde{\mathrm{U}}\xrightarrow{\simeq} \mathrm{B}\Sigma$ is the map of $\Einf$-groups over $\mathrm{BO}$.
\end{proof}

We conclude this appendix by summarizing Stong's computation of the homotopy groups of $\mathrm{M}\Sigma$ \cite{Stong}. To the best of our knowledge, an explicit description of the ring structure of $\pi_*(\mathrm{M}\Sigma)$ is not known (similarly to $\pi_*(\MSU)$).

\begin{theorem}\label{theorem:homotopy_groups_MSigma}
    \begin{enumerate}
        \item The coefficient ring $\pi_*(\mathrm{M}\Sigma)$ contains no odd torsion. 
        \item The map $\pi_*(\MSU)\to \pi_*(\mathrm{M}\Sigma)$ induces an isomorphism on the $2$-primary torsion subgroups.
        \item The ring homomorphism $\pi_*(\mathrm{M}\Sigma)\to \pi_*(\MU)$ becomes an isomorphism after inverting $2$. 
    \end{enumerate}
\end{theorem}
\begin{proof}
    \cite[Theorems 1 and 2]{Stong}
\end{proof}

\section{Real \'etale realization of \texorpdfstring{$\Einf$}{E∞}-ring spectra defined over \texorpdfstring{$\Z$}{Z}}\label{appendix_real_etale}
In this appendix, we recall basic facts about the real \'etale topology and outline the construction of the real \'etale realization following \cite{Bachmann_ret}. Then we compute it for motivic $\mathbb{E}_\infty$-ring spectra that are defined over $\Spec(\Z)$. Consequently, we obtain formulas for the real \'etale realizations of some algebraic cobordism spectra.

Recall that the real spectrum $\Sper(A)$ of a commutative ring $A$ is the set of pairs $(\mathfrak{p},\alpha)$, where $\mathfrak{p}\in\Spec(A)$ and $\alpha$ is an ordering of the residue field $\kappa(\mathfrak{p})$. This set can be equipped with the Harrison topology. More generally, for a scheme $S$, there is a topological space $RS$, which is obtained by gluing real spectra and satisfies $R\Spec(A)=\Sper(A)$. A family of morphisms $\{U_i\to S\}_{i\in I}$ is said to be a \textit{real \'etale covering} if each morphism $U_i\to S$ is \'etale and $RS=\bigcup_{i\in I} RU_i$.
\begin{notation}
    For a qcqs scheme $S$, we denote by $S_{\ret}$ the \textit{small real \'etale site}, consisting of the category $\text{\'Et}_S$ of \'etale $S$-schemes with the real \'etale topology on it, and by $\Sm_{S,\ret}$ the big real \'etale site (i.e., $\SmS$ with the real \'etale topology).
\end{notation}
\begin{notation}
    We write $\SH_{\ret}(S)$ for the localization of $\SHS$ at the real \'etale coverings, $\SH^{S^1}_{\ret}(S)$ for the localization of $\SH^{S^1}(S)$ at the real \'etale coverings, $\Sp(S_{\ret})$ for the stable $\infty$-category of sheaves of spectra on $S_{\ret}$, and $\Sp(RS)$ for the stable $\infty$-category of sheaves of spectra on the topological space $RS$. All these categories are equipped with natural symmetric monoidal structures. 
\end{notation}
    Bachmann's theorem (see \cite[Theorem 35]{Bachmann_ret} and \cite[Theorem 4.2]{BH-rmk-ret}) states that the following natural symmetric monoidal functors are equivalences:
    \begin{equation}
        \Sp(S_{\ret})\xrightarrow{i^*} \SH^{S^1}_{\ret}(S)\xrightarrow{\sigma^\infty_{\ret}} \SH_{\ret}(S).
    \end{equation}
    Here the first functor is induced by the morphism of sites $i\colon S_{\ret}\hookrightarrow \Sm_{S,\ret}$ (see \cite[Theorem 8]{Bachmann_ret}), and the second one is the $\Gm$-stabilization.
\begin{definition}
    The \textit{real \'etale realization functor} is the symmetric monoidal functor given by the following composition 
    \[ \realiz_{\ret}\colon\SHS \xrightarrow{\mathrm{L}_{\ret}} \SH_{\ret}(S) \xrightarrow{\simeq} \Sp(S_{\ret}), \]
    where the last functor is the inverse of Bachmann's equivalence stated above. There is a canonical equivalence $\Sp(S_{\ret})\cong \Sp(RS)$, which is the stabilization of the Elmanto--Shah--Scheiderer equivalence of the underlying $\infty$-topoi; see \cite[Theorem B.10]{ES-ret-topos} and \cite[Theorem 1.3]{Sch-ret-topos}. Composing this equivalence with the realization functor above, we obtain the symmetric monoidal real \'etale realization functor (we use the same name and notation for it)
    \[ \realiz_{\ret}\colon\SHS \to \Sp(RS). \]
\end{definition}
\begin{notation}
    For a scheme $S$, we denote by $\underline{\mathrm{E}}\in\Sp(RS)$ the constant sheaf associated with a spectrum $\mathrm{E}\in\Sp$. The assignment $\mathrm{E}\mapsto \underline{\mathrm{E}}$ preserves $\Einf$-rings since it is defined as the composition of the base change along $RS\to *$ with the sheafification functor.
\end{notation}
\begin{theorem}\label{thm:real_et_realiz}
    Suppose that $A$ is a subring of $\R$ such that $\Sper(A)=\Sper(\R)=*$ (e.g., $A=\Z$), $\EE$ is a motivic spectrum over $A$, and $S\in\mathrm{Sch}_A$. Then the real \'etale realization of $\EE_S$ is given by the constant sheaf associated with the real Betti realization of $\EE$
    \[ \realiz_{\ret}(\EE_S)\cong \underline{\realiz_{\mathrm{B}\R}(\EE_\R)}\in \Sp(RS). \]
    Moreover, if $\EE\in \mathrm{CAlg}(\SH(A))$, then the isomorphism is an equivalence of sheaves of $\mathbb{E}_\infty$-ring spectra.
\end{theorem}
\begin{proof}
    First, consider the following commutative diagram:
    \begin{center}
        \begin{tikzcd}
            \SH(A) \arrow[rr, "\realiz_{\ret}"'] \arrow[d]              &  & \Sp(\Sper(A)) \arrow[d, "\simeq"]                     \\
            \SH(\R) \arrow[rr, "\realiz_{\ret}\simeq\realiz_{\mathrm{B}\R}"'] &  & \Sp(\Sper(\R))\simeq\Sp.
        \end{tikzcd}
    \end{center}
    Here the vertical functors are base changes along $A\hookrightarrow\R$ and $\Sper(\R)\xrightarrow{=}\Sper(A)$ and the horizontal functors are given by the real \'etale realizations. The real \'etale realization over $\R$ is canonically equivalent to the real Betti realization by \cite[Proposition 36]{Bachmann_ret}. Hence, we have 
    \[ \realiz_{\ret}(\EE)\cong \realiz_{\mathrm{B}\R}(\EE_{\R})\in \Sp. \]
    Now denote by $\pi_S\colon S\to \Spec(A)$ the structure morphism of $S$ and consider the following diagram:
    \begin{center}
        \begin{tikzcd}
            \SH(A) \arrow[rr, "\realiz_{\ret}"'] \arrow[dd, "\pi_S^*"] &  & \Sp(\Sper(A)) \arrow[dd, "\pi_{RS}^*"] \arrow[rr, "\simeq", "\mathrm{forgetful}"'] &  & \PShv(\Sper(A),\Sp)\simeq\Sp \arrow[dd, "\pi_{RS}^*"]                              \\
                                                           &  &                                                                                 &  &                                                      \\
            \SH(S) \arrow[rr, "\realiz_{\ret}"']                       &  & \Sp(RS)                                                                &  & \PShv(RS,\Sp), \arrow[ll, "\mathrm{sheafification}"]
        \end{tikzcd}
    \end{center}
    where $\PShv(X,\Sp)$ denotes the stable $\infty$-category of presheaves of spectra on a topological space $X$.
    The left square obviously commutes and the right square commutes  by the assumption on $\Sper(A)$. The right vertical base change functor sends a spectrum to the associated constant presheaf. The desired equivalence follows immediately. It automatically preserves the $\mathbb{E}_\infty$-ring structure, since all functors in both diagrams are symmetric monoidal.
\end{proof}
\begin{example}\label{example_mso}
    Let $S$ be a qcqs scheme. Then the real \'etale realizations of the standard algebraic cobordism spectra are given by the following constant sheaves of $\Einf$-ring spectra on $RS$:
    \[ \realiz_{\ret}(\MGL_S)\cong \underline{\mathrm{MO}}, \hspace{4em} \realiz_{\ret}(\MSL_S)\cong \underline{\mathrm{MSO}}, \hspace{4em} \realiz_{\ret}(\MSp_S)\cong\underline{\MU}. \]
    This follows from Theorem \ref{thm:real_et_realiz} and the computations of the corresponding real Betti realizations (see \cite[Theorem A]{MG_realization}), since these algebraic cobordism spectra are defined over $\Spec(\Z)$.
\end{example}
\begin{remark}
    Note that the above situation is radically different from that of the $l$-adic \'etale realization; see \cite[\S 4]{BBX-Galois} for the case of $\MGL$. The reason is the non-triviality of the small \'etale site of $\Spec(\Z[\nicefrac{1}{l}])$ (in contrast to the small real \'etale site of $\Spec(\Z)$). 
\end{remark}

\section{Geometric diagonals of \texorpdfstring{$\MSL$}{MSL} and \texorpdfstring{$\MML$}{MML} in characteristic 2}\label{appendix_char2}
In this appendix, we compute the geometric diagonals of the homotopy groups of $\MSL$ and $\MML$ over a field $F$ of characteristic $2$ after inverting $2$. The case of $\MSL$ was excluded in \cite{Egor} because of the reference \cite{BHop}. Here we replace it by using the real \'etale realization, taking advantage of working away from $2$.

First, recall that away from $2$, the homotopy group $\pi_{p,q}(\EE)$ of a motivic spectrum $\EE$ (over any base) splits into the product (see e.g., \cite[Remark 4]{ALP}):
\begin{equation}\label{eq:decomp_pm_parts}
    \pi_{p,q}(\EE)[\nicefrac{1}{2}]\cong \pi_{p,q}(\EE)[\nicefrac{1}{2}]^+\times \pi_{p,q}(\EE)[\nicefrac{1}{2}]^-.
\end{equation} 
Here the plus (resp. minus) part of the homotopy group is by definition the corresponding homotopy group of the plus (resp. minus) part of $\EE$. 
Now we recall some notations from \cite{Egor} that are used below.

\begin{notation}
    We denote by $\mathrm{MWL}$ the $c_1$-spherical algebraic cobordism spectrum (see \cite[\S 2]{Egor}), which is equivalent to the $\eta$-cofiber of $\MSL$, by $d\colon \mathrm{MWL}\to \Sigma^{2,1}\MSL$ the boundary morphism in the cofiber sequence \cite[Theorem A(1)]{Egor}, and by $\delta\colon \mathrm{MWL}\to \Sigma^{2,1}\mathrm{MWL}$ the composite of $d$ with the $(2,1)$-suspension of the forgetful morphism $\mathrm{forg}\colon\MSL\to \mathrm{MWL}$.
\end{notation}

\begin{notation}
    By construction, $\delta^2=\Sigma^{2,1}\delta\circ\delta$ is null-homotopic. We denote by $Z_{i,j}(\mathrm{MWL})$ (resp. $H_{i,j}(\mathrm{MWL})$) the group of cycles (resp. homology) of the chain complex $(\pi_{i,j}(\mathrm{MWL}),\delta_*)$ in the degree corresponding to $\pi_{i,j}(\mathrm{MWL})$.
\end{notation}

\begin{lemma}\label{lemma:eta_stab_char2}
    Let $R$ be a local Dedekind domain (i.e., a field or a discrete valuation ring) with residue characteristic $2$. Then the multiplication with the motivic Hopf element $\eta$ induces an isomorphism
    $$\eta\cdot\pi_{2n-1,n-1}(\MSL)[\nicefrac{1}{2}]\xrightarrow{\simeq} \eta\cdot\pi_{2n,n}(\MSL)[\nicefrac{1}{2}].$$
    Moreover, the canonical map $\MSL\to \MSL[\eta^{-1}]$ identifies the right group with $\pi_{2n+1,n+1}(\MSL)[\nicefrac{1}{2}]^-$.
\end{lemma}
\begin{proof}
    By \cite[Lemma 6.2]{Egor}, we have an exact sequence:
    \[ H_{2n+2,n+1}(\mathrm{MWL})\to \eta\cdot\pi_{2n-1,n-1}(\MSL)\xrightarrow{\eta} \eta\cdot\pi_{2n,n}(\MSL)\to H_{2n,n}(\mathrm{MWL}). \]
    After inverting $2$, the left and right groups vanish by \cite[Theorem 4.14 and Proposition B.9]{Egor}. This proves the first statement. The resulting group coincides with the corresponding homotopy group of the $\eta$-periodization by \cite[Theorem 5.2]{Egor}.
\end{proof}

\begin{lemma}\label{lemma:kernels_char2}
    Let $R$ be a local Dedekind domain with residue characteristic $2$. Then the inclusion
    $$ \mathrm{Ker}(\pi_{2n,n}(d))\subset Z_{2n,n}(\mathrm{MWL}) $$
    becomes an equality after inverting $2$.
\end{lemma}
\begin{proof}
    We need to show that the image of $d_*[\nicefrac{1}{2}]\colon\pi_{2n,n}(\mathrm{MWL})[\nicefrac{1}{2}]\to \pi_{2n-2,n-1}(\MSL)[\nicefrac{1}{2}]$ does not intersect the kernel of $\mathrm{forg}_*[\nicefrac{1}{2}]\colon\pi_{2n-2,n-1}(\MSL)[\nicefrac{1}{2}]\to \pi_{2n-2,n-1}(\mathrm{MWL})[\nicefrac{1}{2}]$. By the exact sequence \cite[(5.1)]{Egor}, we have an equality $\mathrm{Ker}(\mathrm{forg}_*)=\eta\cdot\pi_{2n-3,n-2}(\MSL)$. After tensoring with $\Z[\nicefrac{1}{2}]$, the latter group is equal to the corresponding homotopy group of the $\eta$-periodization by Lemma \ref{lemma:eta_stab_char2}. In turn, the image of $d_*[\nicefrac{1}{2}]$ consists of $\eta$-torsion elements, since the motivic Hopf element is trivial in the $c_1$-spherical algebraic cobordism spectrum, see \cite[Theorem A]{Egor}. Hence, the intersection $\mathrm{Im}(d_*[\nicefrac{1}{2}])\cap \mathrm{Ker}(\mathrm{forg}_*[\nicefrac{1}{2}])$ is trivial and the result follows.
\end{proof}

\begin{lemma}\label{lemma:plus_part_char2}
    Let $R$ be a discrete valuation ring of mixed characteristics $(0,2)$ with residue field $F$. Then base change along the span $L:=\mathrm{Frac}(R)\leftarrow R\twoheadrightarrow F$ induce isomorphisms
    \begin{align*} 
    & \pi_{2*,*}(\MSL_L)[\nicefrac{1}{2}]^+\cong \pi_{2*,*}(\MSL_{F})[\nicefrac{1}{2}]^+, \\
    & \pi_{2*,*}(\MML_L)[\nicefrac{1}{2}]^+\cong \pi_{2*,*}(\MML_{F})[\nicefrac{1}{2}]^+.
    \end{align*}
\end{lemma}
\begin{proof}
    First, we deal with $\MSL$. Below $A$ denotes one of the rings $L$, $R$, or $F$. 
    Consider the following diagram (see \cite[(6.1)]{Egor} for the exact row):
    \[
    \begin{tikzcd}[column sep=small]
    0 \ar[r] & 
    \eta\cdot\pi_{2n-1,n-1}(\mathrm{MSL}_A) [\nicefrac{1}{2}] \ar[r] \ar[d, "\simeq", "\eta"'] & 
    \pi_{2n,n}(\MSL_A)[\nicefrac{1}{2}] \ar[r] \ar[ld, "\eta"] & 
    \mathrm{Ker} (\pi_{2n,n}(d))[\nicefrac{1}{2}] \ar[r] & 0 \\
    & 
    \eta\cdot\pi_{2n,n}(\mathrm{MSL}_A)[\nicefrac{1}{2}]. 
    \end{tikzcd}
    \]
    Here the vertical map is an isomorphism by Lemma \ref{lemma:eta_stab_char2} and the right corner is equal to the group $Z_{2n,n}(\mathrm{MWL}_A)[\nicefrac{1}{2}]$ by Lemma \ref{lemma:kernels_char2}. The diagonal map gives a splitting of the short exact sequence and we obtain 
    $$ \pi_{2n,n}(\MSL_A)[\nicefrac{1}{2}]\cong Z_{2n,n}(\mathrm{MWL}_A)[\nicefrac{1}{2}]\times \eta\cdot\pi_{2n-1,n-1}(\MSL_A)[\nicefrac{1}{2}]. $$
    This decomposition coincides with the decomposition of $\pi_{2n,n}(\MSL_A)[\nicefrac{1}{2}]$ into the product of the plus and minus parts. Hence, the desired plus part over $A$ is isomorphic to the group $Z_{2n,n}(\mathrm{MWL}_A)[\nicefrac{1}{2}]$, which is stable under base changes along the homomorphisms $R\to L$ and $R\to F$ by \cite[Theorem~4.14]{Egor}.

    Now we consider the case of $\MML$. The cofiber sequence from Theorem \ref{thm:main_cofib_seq} induces the short exact sequence (note that $\Sigma^1_{\Proj^1}\MGL_A[\nicefrac{1}{2}]$ is canonically equivalent to its plus part):
    $$ 0\to \pi_{2n,n}(\MSL_A)[\nicefrac{1}{2}]^+\to \pi_{2n,n}(\MML_A)[\nicefrac{1}{2}]^+\to \pi_{2n-2,n-1}(\MGL_A)[\nicefrac{1}{2}]\to 0.$$
    The result follows by taking the base change of this sequence along the span $L\leftarrow R\to F$ and using $\pi_{2*,*}(\MSL_A)[\nicefrac{1}{2}]\cong \mathbb{L}[\nicefrac{1}{2}]$ \cite[Corollary 6.6]{SpiMGL} together with the result already proven for $\MSL$.
\end{proof}

\begin{theorem}\label{theorem:geom_diags_char2}
    Suppose that $F$ is a field of characteristic $2$. Then there are isomorphisms of $\GW(F)$-algebras:
    \begin{align*}
        & \pi_{2*,*}(\MSL)[\nicefrac{1}{2}]\cong \pi_{2*}(\MSU)[\nicefrac{1}{2}], \\
        & \pi_{2*,*}(\MML)[\nicefrac{1}{2}]\cong \pi_{2*}(\mathrm{M}\Sigma)[\nicefrac{1}{2}].
    \end{align*}
\end{theorem}
\begin{proof}
    Consider a discrete valuation ring $R$ of mixed characteristic $(0,2)$ whose residue field is $F$ (see the proof of Lemma \ref{lemma:quotient_by_fund_ideals}). Then base change along the span $\mathrm{Frac}(R)\leftarrow R\to F$ reduces the plus parts to the corresponding plus parts over $\mathrm{Frac}(R)$ by Lemma \ref{lemma:plus_part_char2}. They are isomorphic to the topological counterparts by \cite[Theorem D]{Egor} and Lemma \ref{lemma:quotient_by_fund_ideals}. Hence, it remains to show that the minus parts of the desired geometric diagonals are trivial. In fact, it follows from Bachmann's equivalence \eqref{equa_bach_equiv_minus} that the minus part of any motivic spectrum over $F$ is zero, since $\Sper(F)=\emptyset$.
\end{proof}

\begin{remark}
    Formally, the above isomorphisms can be viewed as parts of the cartesian squares stated in \cite[Theorem D]{Egor} and Theorem \ref{theorem_C} away from the characteristic. Indeed, in characteristic $2$, after inverting $2$, the bottom rings (in both cases) are trivial ($\mathrm{W}(F)$ is a $2$-primary torsion group by \cite[Proposition 31.4(6)]{EKM}). For this reason, we do not distinguish the case of characteristic $2$ from the general answer, even though this case turned out to be degenerate.
\end{remark}

\begin{corollary}\label{corollary:MML_MGL_char2}
    Let $F$ be a field of characteristic $2$. Then the canonical morphism $\MML\to \MGL$ induces an isomorphism
    \[ \pi_{2*,*}(\MML)[\nicefrac{1}{2}]\cong \pi_{2*,*}(\MGL)[\nicefrac{1}{2}]. \]
\end{corollary}
\begin{proof}
    The same as in the proof of Corollary \ref{corollary:geom_diag_away_2} using Lemma \ref{lemma:plus_part_char2} instead of Lemma \ref{lemma:rigidity} and Remark \ref{remark:rigidity_dvr}.
\end{proof}

\bibliographystyle{amsalpha}
\bibliography{MML}

@article{BH-Norms,
    author = {Bachmann, T. AND Hoyois, M.},
    title = {Norms in motivic homotopy theory},
    journal = {Ast\'{e}risque},
    Volume = {425},
    year = {2021},
    pages = {208 pp.},
    note = {\href{https://doi.org/10.24033/ast.1147}{doi:10.24033/ast.1147}}
}

@incollection{AnSL,
    author = {Ananyevskiy, A.},
    title = {{$SL$}-oriented cohomology theories},
    booktitle = {Motivic Homotopy Theory and Refined Enumerative Geometry},
    publisher = {Contemp. Math.},
    year = {2020},
    volume = {745},
    pages = {1--20},
    note = {\href{https://doi.org/10.1090/conm/745}{doi:10.1090/conm/745}}
}

@article{Haut, 
    author = {Haution, O.},
    title = {Odd rank vector bundles in eta-periodic motivic homotopy theory}, 
    DOI = {10.1017/S1474748023000294}, 
    journal = {J. Inst. Math. Jussieu}, 
    year = {2025},
    volume = {24},
    number = {5},
    pages = {pp. 1733--1764},
    note = {\href{https://doi.org/10.1017/S1474748023000294}{doi:10.1017/S1474748023000294}}
}

@article{HJNY-HermKtheory,
    author = {Hoyois, M. AND Jelisiejew, J. AND Nardin, D. AND Yakerson, M.},
    title = {Hermitian {K}-theory via oriented {G}orenstein algebras},
    journal = {J. Reine Angew. Math.},
    volume = {793},
    year = {2022},
    pages = {pp. 105--142},
    note = {\href{https://doi.org/10.1515/crelle-2022-0063}{doi:10.1515/crelle-2022-0063}}
}

@unpublished{Ahina,
    author = {Nandy, A.},
    title = {An interpolation between special linear and general algebraic cobordism {$MSL$} and {$MGL$}},
    year = {2024},
    note = {preprint, \href{https://arxiv.org/abs/2310.15721}{arXiv:2310.15721}}
}

@article{Egor,
    author = {Zolotarev, E.},
    title = {The geometric diagonal of the special linear algebraic cobordism},
    journal = {Adv. Math.},
    year = {2026},
    pages = {Paper no. 110922, 60 pp.},
    volume = {493},
    note = {\href{https://doi.org/10.1016/j.aim.2026.110922}{doi:10.1016/j.aim.2026.110922}}
}

@article{Voev98,
    author = {Voevodsky, V.},
    title = {$\mathbb{A}^1$-homotopy theory},
    journal = {Doc. Math.},
    volume = {Extra Vol. I},
    year = {1998},
    pages = {pp. 579--604}
}

@article{PW-BO,
    author = {Panin, I. AND Walter, C.},
    title = {On the motivic commutative ring spectrum {$BO$}},
    journal = {St. Petersburg Math. J.},
    year = {2018},
    volume = {30},
    number = {6},
    pages = {pp. 933--972},
    note = {\href{https://doi.org/10.1090/spmj/1578}{doi:10.1090/spmj/1578}}
}

@article{PW-Thom,
    author = {Panin, I. AND Walter, C.},
    title = {On the algebraic cobordism spectra {$MSL$} and {$MSp$}},
    journal = {St. Petersburg Math. J.},
    year = {2023},
    volume = {34},
    number = {1},
    pages = {pp. 144--187},
    note = {\href{https://doi.org/10.1090/spmj/1748}{doi:10.1090/spmj/1748}}
}

@unpublished{MSL-slices,
    author = {Nandy, A. AND R\"{o}ndigs, O. AND Zolotarev, E.},
    title = {Slices of the special linear algebraic cobordism spectrum},
    year = {2026},
    note = {preprint, \href{https://arxiv.org/abs/2606.05020}{arXiv:2606.05020}}
}

@article{Mb-modules, 
    title = {Modules over algebraic cobordism}, 
    volume = {8},
    journal = {Forum Math. Pi}, 
    author = {Elmanto, E. and Hoyois, M. and Khan, A. A. and Sosnilo, V. and Yakerson, M.}, 
    year = {2020}, 
    pages = {e14, 44 pp.},
    note = {\href{https://doi.org/10.1017/fmp.2020.13}{doi:10.1017/fmp.2020.13}}
}

@book{Morel_book,
  title={$\mathbb{A}^1$-Algebraic Topology over a Field},
  author={Morel, F.},
  isbn={9783642295157},
  series={Lecture Notes in Mathematics},
  url={https://books.google.de/books?id=ptYsjgEACAAJ},
  year={2012},
  publisher={Springer Berlin Heidelberg}
}

@article{BW_Euler, 
    title={EULER CLASSES: SIX-FUNCTORS FORMALISM, DUALITIES, INTEGRALITY AND LINEAR SUBSPACES OF COMPLETE INTERSECTIONS},
    volume={22}, 
    number = {2},
    journal={J. Inst. Math. Jussieu},
    author={Bachmann, T. and Wickelgren, K.}, 
    year={2023}, 
    pages={pp. 681--746},
    note={\href{https://doi.org/10.1017/s147474802100027x}{doi:10.1017/s147474802100027x}}
}

@article{MS_Cellular,
    title = {Cellular $\mathbb{A}^1$-homology and the motivic version of {M}atsumoto's theorem},
    journal = {Adv. Math.},
    volume = {434},
    pages = {Paper no. 109346, 110 pp.},
    year = {2023},
    author = {Morel, F. and Sawant, A.},
    note = {\href{https://doi.org/10.1016/j.aim.2023.109346}{doi:10.1016/j.aim.2023.109346}}
}

@article{AHW_BG,
    author = {Asok, A. AND Hoyois, M. AND Wendt, M.},
    title = {Affine representability results in $\mathbb{A}^1$–homotopy theory, {II}: {P}rincipal bundles and homogeneous spaces},
    journal = {Geom. Topol.},
    year = {2018},
    volume = {22},
    number = {2},
    pages = {pp. 1181--1225},
    note = {\href{https://doi.org/10.2140/gt.2018.22.1181}{doi:10.2140/gt.2018.22.1181}}
}

@article{Stong,
    author = {Stong, R. E.},
    title = {On {C}omplex-{S}pin {M}anifolds},
    journal = {Ann. of Math. (2)},
    year = {1967},
    volume = {85},
    number = {3},
    pages = {pp. 526--536},
    note = {\href{https://doi.org/10.2307/1970357}{doi:10.2307/1970357}}
}

@article{Atiyah,
    author = {Atiyah, M. F.},
    title = {Riemann surfaces and spin structures},
    journal = {Ann. Sci. \'Ecole Norm. Sup. (4)},
    year = {1971},
    volume = {4},
    number = {1},
    pages = {pp. 47--62},
    note = {\href{https://doi.org/10.24033/asens.1205}{doi:10.24033/asens.1205}}
}

@unpublished{BrazeltonWendtMSLc,
    author = {Brazelton, T. AND Wendt, M.},
    title = {The {C}how--{W}itt rings of the classifying space of quadratically oriented bundles},
    note = {preprint, \href{https://arxiv.org/abs/2501.0630}{arXiv:2501.0630}},
    year = {2025}
}

@unpublished{Hoyois_plus,
    author = {Hoyois, M.},
    title = {ON {Q}UILLEN’S PLUS CONSTRUCTION},
    year = {2019},
    note =  {note, \href{https://hoyois.app.uni-regensburg.de/papers/acyclic.pdf}{hoyois.app.uni-regensburg.de/papers/acyclic.pdf}}
}

@article{Inf_loop_plus,
    author = {Bachmann, T. AND Elmanto, E. AND Hoyois, M. AND Khan, A. A. AND Sosnilo, V. AND Yakerson, M.},
    title = {On the infinite loop spaces of algebraic cobordism and the motivic sphere},
    journal = {\'Epijournal G\'eom. Alg\'ebrique},
    year = {2021},
    pages = {Article Nr. 9, 13 pp.},
    volume = {5},
    note = {\href{https://doi.org/10.46298/epiga.2021.volume5.6581}{doi:10.46298/epiga.2021.volume5.6581}}
}

@article{McDuff-Segal,
    author = {McDuff, D. AND Segal, G.},
    title = {Homology fibrations and the ``group-completion'' theorem},
    journal = {Invent. Math.},
    year = {1976},
    volume = {31},
    pages = {pp. 279--284},
    note = {\href{https://doi.org/10.1007/BF01403148}{doi:10.1007/BF01403148}}
}

@article{Splitting_lemma,
    author = {Devalapurkar, S. AND Haine, P.},
    title = {On the {J}ames and {H}ilton--{M}ilnor splittings, and the metastable {EHP} sequence},
    journal = {Doc. Math.},
    year = {2021},
    volume  = {26},
    pages = {pp. 1423--1464},
    note = {\href{https://doi.org/10.4171/dm/845}{doi:10.4171/dm/845}}
}

@article{DFJK,
    author = {D{\'e}glise, F. AND Fasel, J. AND Jin, F. AND Khan, A. A.},
    title = {On the rational motivic homotopy category},
    journal = {J. \'Ec. polytech. Math.},
    year = {2021},
    volume = {8},
    pages = {pp. 533--583},
    note = {\href{https://doi.org/10.5802/jep.153}{doi:10.5802/jep.153}}
}

@article{Hoy15,
    author = {Hoyois, M.},
    title = {From algebraic cobordism to motivic cohomology},
    journal = {J. Reine Angew. Math.},
    year = {2015},
    volume = {702},
    pages = {pp. 173--226},
    note = {\href{https://doi.org/10.1515/crelle-2013-0038}{doi:10.1515/crelle-2013-0038}}
}

@article{ARO,
    author = {Ananyevskiy, A. AND R{\"o}ndigs, O. AND {\O}stv{\ae}r, P. A.},
    title = {On very effective hermitian {K}-theory},
    journal = {Math. Z.},
    volume = {294},
    year = {2020},
    pages = {pp. 1021--1034},
    note = {\href{https://doi.org/10.1007/s00209-019-02302-z}{doi:10.1007/s00209-019-02302-z}},
}

@article{SpiMGL,
    author = {Spitzweck, M.},
    title = {Algebraic {C}obordism in mixed characteristic},
    journal = {Homology Homotopy Appl.},
    year = {2020},
    volume = {22},
    number = {2},
    pages = {pp. 91--103},
    note = {\href{https://doi.org/10.4310/HHA.2020.v22.n2.a5}{doi:10.4310/HHA.2020.v22.n2.a5}}
}

@article{Morel_hm,
    author = {Morel, F.},
    title = {An introduction to $\mathbb{A}^1$-homotopy theory},
    journal = {ICTP lecture notes series},
    year = {2003},
    volume = {15},
    pages = {pp. 357--441}
}

@article{Bachmann_ret,
    author = {Bachmann, T.},
    title = {Motivic and real \'etale stable homotopy theory},
    journal = {Compos. Math.},
    year = {2018},
    volume = {154},
    number = {5},
    pages = {pp. 883--917},
    note = {\href{https://doi.org/10.1112/S0010437X17007710}{doi:10.1112/S0010437X17007710}}
}

@article{RSO19,
	Author = {R{\"o}ndigs, O. AND Spitzweck, M. AND {\O}stv{\ae}r, P. A.},
	Journal = {Ann. of Math. (2)},
	Pages = {pp. 1--74},
	Title = {The first stable homotopy groups of motivic spheres},
	Volume = {189},
    number = {1},
	Year = {2019},
    note = {\href{https://doi.org/10.4007/annals.2019.189.1.1}{doi:10.4007/annals.2019.189.1.1}}
}

@unpublished{BHop,
    author = {Bachmann, T. AND Hopkins, M.},
    title = {$\eta$-periodic motivic stable homotopy theory over fields},
    note = {preprint, \href{https://arxiv.org/abs/2005.06778}{arXiv:2005.06778}},
    year = {2021}
}

@unpublished{BBX-Galois,
    author = {Bachmann, T. AND Burklund, R. AND Xu, Z.},
    title = {Motivic stable stems and {G}alois approximations of cellular motivic categories},
    year = {2025},
    note = {preprint, \href{https://arxiv.org/abs/2503.12060}{arXiv:2503.12060}}
}

@article{ES-ret-topos,
    author = {Elmanto, E. AND Shah, J.},
    title = {Scheiderer motives and equivariant higher topos theory},
    journal = {Adv. Math.},
    year = {2021},
    volume = {382},
    pages = {Paper no. 107651, 116 pp.},
    note = {\href{https://doi.org/10.1016/j.aim.2021.107651}{doi:10.1016/j.aim.2021.107651}}
}

@book{Sch-ret-topos,
    author = {Scheiderer, C.},
    title = {Real and \'{E}tale Cohomology},
    series={Lecture Notes in Mathematics},
    publisher = {Springer Berlin, Heidelberg},
    year = {1994},
    note = {\href{https://doi.org/10.1007/BFb0074269}{doi:10.1007/BFb0074269}}
}

@unpublished{BH-rmk-ret,
    author = {Bachmann, T. AND Hoyois, M.},
    title = {Remarks on \'etale motivic stable homotopy theory},
    year = {2021},
    note = {preprint, \href{https://arxiv.org/abs/2104.06002}{arXiv:2104.06002}, to appear in Math. Res. Lett.}
}

@article{EK-perfection,
    author = {Elmanto, E. AND Khan, A. A.},
    title = {Perfection in motivic homotopy theory},
    journal = {Proc. Lond. Math. Soc. (3)},
    year = {2020},
    volume = {120},
    number = {1},
    pages = {pp. 28--38},
    note = {\href{https://doi.org/10.1112/plms.12280}{doi:10.1112/plms.12280}}
}

@article{LYZ,
    author = {Levine, M. AND Yang, Y. AND Zhao, G.},
    title = {Algebraic elliptic cohomology and flops {II}: {$SL$}-cobordism},
    journal = {Adv. Math.},
    year = {2021},
    volume = {384},
    pages = {Paper no. 107726, 66 pp.},
    note = {\href{https://doi.org/10.1016/j.aim.2021.107726}{doi:10.1016/j.aim.2021.107726}}
}

@article{RO16,
    author = {R{\"o}ndigs, O. AND {\O}stv{\ae}r, P. A.},
    title = {Slices of hermitian {K}-theory and {M}ilnor's conjecture on quadratic forms},
    journal = {Geom. Topol.},
    year = {2016},
    volume = {20},
    number = {2},
    pages = {pp. 1157--1212},
    note = {\href{https://doi.org/10.2140/gt.2016.20.1157}{doi:10.2140/gt.2016.20.1157}}
}

@article{MNN-mod-cat,
    author = {Mathew, A. AND Naumann, N. AND Noel, J.},
    title = {Nilpotence and descent in equivariant stable homotopy theory},
    journal = {Adv. Math.},
    year = {2017},
    volume = {305},
    pages = {pp. 994--1084},
    note = {\href{https://doi.org/10.1016/j.aim.2016.09.027}{doi:10.1016/j.aim.2016.09.027}}
}

@unpublished{MG_realization,
    author = {Bannwart, J.},
    title = {The real {B}etti realization of motivic {T}hom spectra and of very effective {H}ermitian {K}-theory},
    year = {2025},
    note = {preprint, \href{https://arxiv.org/abs/2505.07297}{arXiv:2505.07297}}
}

@article{RSO24,
    Author = {R{\"o}ndigs, O. AND Spitzweck, M. AND {\O}stv{\ae}r, P. A.},
	Journal = {Duke Math. J.},
	Pages = {pp. 1017--1084},
	Title = {The second stable homotopy groups of motivic spheres},
	Volume = {173},
    number = {6},
	Year = {2024},
    note = {\href{https://doi.org/10.1215/00127094-2023-0023}{doi:10.1215/00127094-2023-0023}}
}

@incollection{HAZEWINKEL,
    title = {Witt vectors. {P}art 1},
    series = {Handbook of Algebra},
    publisher = {North-Holland},
    volume = {6},
    pages = {319--472},
    year = {2009},
    booktitle = {Handbook of Algebra},
    issn = {1570-7954},
    note = {\href{https://doi.org/10.1016/S1570-7954(08)00207-6}{doi:10.1016/S1570-7954(08)00207-6}},
    author = {Hazewinkel, M.},
}

@article{BU_tilde_physics,
    author = {Forger, M. and Hess, H.},
    title = {{Universal metaplectic structures and geometric quantization}},
    volume = {64},
    number = {3},
    journal = {Comm. Math. Phys.},
    publisher = {Springer},
    pages = {pp. 269--278},
    year = {1979},
    note = {\href{https://doi.org/10.1007/BF01221734}{doi:10.1007/BF01221734}}
}

@article{PPR08,
    author = {Panin, I. AND Pimenov, K. AND R{\"o}ndigs, O.},
    title = {A universality theorem for {V}oevodsky's algebraic cobordism spectrum},
    journal = {Homology Homotopy Appl.},
    year = {2008},
    volume = {10},
    number = {2},
    pages = {pp. 211--226},
    note = {\href{http://doi.org/10.4310/HHA.2008.v10.n2.a11}{doi:10.4310/HHA.2008.v10.n2.a11}}
}

@article{Quillen,
    author = {Quillen, D.},
    title = {On the formal group laws of unoriented and complex cobordism theory},
    journal = {Bull. Amer. Math. Soc.},
    year = {1969},
    volume = {75},
    number = {6},
    pages = {pp. 1293--1298}
}

@article{Milnor_Witt_motives,
    author = {Bachmann, T. AND Calm\'es, B. AND D\'eglise, F. AND Fasel, J. and {\O}stv{\ae}r, P. A.},
    title = {Milnor--{W}itt {M}otives},
    journal = {Mem. Amer. Math. Soc.},
    year = {2025},
    pages = {vii+201 pp.},
    volume = {311},
    note = {\href{https://doi.org/10.1090/memo/1572}{doi:10.1090/memo/1572}}
}

@article{Einf-semirings,
    author = {Gepner, D. AND Groth, M. and Nikolaus, T.},
    title = {Universality of multiplicative infinite loop space machines},
    journal = {Algebr. Geom. Topol.},
    year = {2015},
    volume = {15},
    number = {6},
    pages = {pp. 3107--3153},
    note = {\href{http://dx.doi.org/10.2140/agt.2015.15.3107}{doi:10.2140/agt.2015.15.3107}}
}

@unpublished{LHA,
    author = {Lurie, J.},
    title = {Higher {A}lgebra},
    year = {2017},
    note = {book \href{http://www.math.harvard.edu/~lurie/papers/HA.pdf}{math.harvard.edu/~lurie/papers/HA.pdf}}
}

@article{Levine,
    author = {Levine, M.},
    title = {Comparison of cobordism theories},
    journal = {J. Algebra},
    year = {2009},
    volume = {322},
    number = {9},
    pages = {pp. 3291--3317},
    note = {\href{https://doi.org/10.1016/j.jalgebra.2009.03.032}{doi:10.1016/j.jalgebra.2009.03.032}}
}

@unpublished{Hoyois_Land,
    author = {Hoyois, M. AND Land, M.},
    title = {Grothendieck--{W}itt theory of derived schemes},
    year = {2025},
    note = {preprint, \href{https://arxiv.org/abs/2508.08905}{arXiv:2508.08905}}
}

@book{EKM,
    author = {Elman, R. AND Karpenko, N. AND Merkurjev, A.},
    title = {The {A}lgebraic and {G}eometric Theory of {Q}uadratic {F}orms},
    publisher = {Colloquium Publications},
    year = {2008},
    volume = {58},
    pages = {435}
}

@unpublished{KRO_hermitian_k_theory_char2,
    author = {Kolderup, H. AND R\"ondigs, O. AND {\O}stv{\ae}r, P. A.},
    title = {Hermitian {K}-theory and {M}ilnor--{W}itt motivic cohomology over $\mathbb{Z}$},
    year = {2025},
    note = {preprint, \href{https://arxiv.org/abs/2509.16404v1}{arXiv:abs/2509.16404v1}}
}

@article{ALP,
    author = {Ananyevskiy, A. AND Levine, M. AND Panin, I.},
    title = {Witt sheaves and the $\eta$-inverted sphere spectrum},
    journal = {J. Topol.},
    year = {2017},
    volume = {10},
    number = {2},
    pages = {pp. 370--385},
    note = {\href{https://doi.org/10.1112/topo.12015}{doi:10.1112/topo.12015}}
}

@article{NSOLandw,
    author = {Naumann, N. AND Spitzweck, M. AND {\O}stv{\ae}r, P. A.},
    title = {Motivic {L}andweber exactness},
    journal = {Doc. Math.},
    year = {2009},
    volume = {14},
    pages = {pp. 551-593},
    note = {\href{https://doi.org/10.4171/dm/282}{doi:10.4171/dm/282}}
}

@book{LHTT,
    author = {Lurie, J.},
    title = {Higher topos theory},
    series = {Ann. Math. Stud.},
    volume = {170},
    publisher = {Princeton, NJ: Princeton University Press},
    year = {2009},
    note = {\href{https://doi.org/10.1515/9781400830558}{doi:10.1515/9781400830558}}
}

@unpublished{Voev_MC,
    author = {Voevodsky, V.},
    title = {The {M}ilnor {C}onjecture},
    year = {1996},
    note = {preprint}
}

@article{Voevodsky_MC,
    author = {Voevodsky, V.},
    title = {Motivic cohomology with $\mathbb{Z}/2$-coefficients},
    journal = {Publ. Math. Inst. Hautes \'Etudes Sci.},
    year = {2003},
    volume = {98},
    pages = {pp. 59--104},
    note = {\href{https://doi.org/10.1007/s10240-003-0010-6}{doi:10.1007/s10240-003-0010-6}}
}

@article{Oliver_theta,
    author = {R{\"o}ndigs, O.},
    title = {THETA CHARACTERISTICS AND STABLE HOMOTOPY TYPES OF CURVES},
    journal = {Q. J. Math.},
    year = {2010},
    volume = {61},
    number = {3},
    pages = {pp. 351--362},
    note = {\href{https://doi.org/10.1093/qmath/hap005}{doi:10.1093/qmath/hap005}}
}

@article{CLP,
	Author = {Chernykh, G. AND Limonchenko, I. AND Panov, T.},
	Journal = {Russ. Math. Surv.},
	Pages = {pp. 461--524},
	Title = {{SU}-bordism: structure results and geometric representatives},
	Volume = {74},
    Number = {3},
	Year = {2019},
    note = {\href{https://doi.org/10.1070/rm9883}{doi:10.1070/rm9883}}
}

@article{Hilb_K_theory,
    author = {Hoyois, M. AND Jelisiejew, J. AND Nardin, D. AND Totaro, B. AND Yakerson, M.},
    title = {The {H}ilbert scheme of infinite affine space and algebraic {K}-theory},
    journal = {J. Eur. Math. Soc. (JEMS)},
    year = {2025},
    volume = {27},
    number = {5},
    pages = {pp. 2037--2071},
    note = {\href{https://doi.org/10.4171/jems/1340}{doi:10.4171/jems/1340}}
}

@article{Levine_coniveau,
    author = {Levine, M.},
    title = {The homotopy coniveau tower},
    journal = {J. Topol.},
    volume = {1},
    number = {1},
    pages = {pp. 217--267},
    doi = {https://doi.org/10.1112/jtopol/jtm004},
    note = {\href{https://doi.org/10.1112/jtopol/jtm004}{doi:10.1112/jtopol/jtm004}},
    year = {2008}
}

@inproceedings{Voev_slices,
    author = {Voevodsky, V.},
    title = {A possible new approach to the motivic spectral sequence for algebraic {K}-theory},
    booktitle = {Recent progress in homotopy theory},
    year = {2002},
    volume = {293},
    pages = {371--379},
    series = {Contemp. Math.},
    publisher = {Amer. Math. Soc., Providence, RI}
}

@article{Oliver_Suslin,
    author = {R{\"o}ndigs, O.},
    title = {Endomorphisms of the projective plane and the image of the {S}uslin--{H}urewicz map},
    journal = {Invent. Math.},
    volume = {232},
    number = {3},
    pages = {pp. 1161--1194},
    year = {2023},
    note = {\href{https://doi.org/10.1007/s00222-023-01179-4}{doi:10.1007/s00222-023-01179-4}}
}

\end{document}